\documentclass[12pt]{amsart}
\usepackage[utf8]{inputenc}
\usepackage{latexsym,amsmath,amssymb,amsmath,mathtools,tikz,
microtype,comment,amsthm,soul}

\usepackage[english]{babel}
\usepackage[margin=2.5cm]{geometry}
\usepackage{enumitem}

\usepackage{tikz-cd}
\usepackage{mathrsfs}
\usetikzlibrary{calc,intersections,through,backgrounds,patterns}
\usetikzlibrary{decorations.markings}
\usetikzlibrary{decorations.pathreplacing}

\usepackage{pifont}
\usepackage{yfonts}

\usepackage{tikz}

\usepackage[colorlinks=true, citecolor=blue, urlcolor=magenta, breaklinks=true]{hyperref}

\numberwithin{equation}{section}
\newtheorem{theorem}{Theorem}[section]
\newtheorem{lemma}[theorem]{Lemma}

\newtheorem{corollary}[theorem]{Corollary}

\theoremstyle{definition}
\newtheorem{definition}[theorem]{Definition} 
\newtheorem{remark}[theorem]{Remark}
\newtheorem{example}[theorem]{Example}

\newtheorem{convention}[theorem]{Convention}

\title{Splittings of ideals of points in $\mathbb{P}^{1}\times\mathbb{P}^{1}$}

\author{Elena Guardo}
\address[E. Guardo]{Dipartimento di Matematica e Informatica, Viale A. Doria 6, 95125, Catania}
\email{elena.guardo@unict.it}

\author{Graham Keiper}
\address[G. Keiper]{Dipartimento di Matematica e Informatica, Viale A. Doria 6, 95125, Catania}
\email{gtrkeiper@outlook.com}

\author{Adam Van Tuyl}
\address[A. Van Tuyl]
{Department of Mathematics and Statistics
McMaster University, Hamilton, ON L8S 4L8, Canada}
\email{vantuyla@mcmaster.ca}


\date{\today}
\begin{document}

\begin{abstract}
Let $I_\mathbb{X}$ be the bihomogeneous ideal
of a finite set of points $\mathbb{X} \subseteq
\mathbb{P}^1 \times \mathbb{P}^1$.  The
purpose of this note is to consider ``splittings''
of the ideal $I_\mathbb{X}$, that is, finding ideals $J$ and $K$
such that $I_\mathbb{X} = J+K$, where $J$ and $K$ have prescribed
algebraic or geometric properties. We show that for any set
of points $\mathbb{X}$, we cannot partition
the generators of $I_\mathbb{X}$ into two ideals of points.
The best case scenario is where at most one of 
$J$ or $K$ is an ideal of points.  To remedy this we introduce the notion of unions of lines and ACM (Arithmetically
Cohen-Macaulay) points which allows us to say more about splittings.
For a set $\mathbb{W}$ of unions of lines and ACM sets of points, we can
write $I_\mathbb{W} = J + K$ where both $J$ and $K$ are ideals of unions of lines and ACM points as well.  
When $\mathbb{W}$ is a union of lines and ACM points, we  
discuss some consequences for the 
graded Betti numbers of $I_{\mathbb{W}}$ in terms of these  splittings.
\end{abstract}
\maketitle

%%%%%%%%%%%%%%%%%%%%%%%%%%%%%%%%%%%%%%%%%%%%%%%%%%%%%%%%
\section{Introduction}\label{sec:intro}

Let $I$ be a homogeneous ideal of a polynomial ring
$R = \mathbb{K}[x_1,\ldots,x_n]$ over a field
$\mathbb{K}$ of characteristic zero.  If $\mathcal{G}(I) = \{g_1,\ldots,g_t\}$ is a minimal
set of generators of $I$, we say $I = J+K$ is a {\it splitting} of $I$
if there is  subset $A \subseteq [t]:=\{1,\ldots,t\}$
such that $J = \langle g_i ~|~ i \in A \rangle$ and $K =
\langle g_i ~|~ i \in [t] \setminus A \rangle$.  In 
other words, $J$ and $K$ are formed by partitioning the set of
generators of $I$.  While it is clear that there are many ways to
``split'' a given ideal $I$, some splittings are more distinguished 
in the sense 
that the ideals $J$ and $K$ may provide algebraic or 
geometric information about $I$.  In some cases, these splittings lend themselves to inductive arguments or allow us to iteratively construct new ideals with specified properties (for example, see \cite{FKVT2020}). 

A classical example of this idea is the work of
Eliahou and Kervaire \cite{EK90} where they split a stable 
monomial ideal  $I = J+K$
into two smaller stable monomial ideals $J$ and $K$, such 
that the graded Betti numbers of $I$ can be recursively 
computed in terms of those $J, K$ and $J \cap K$.
Inspired by Eliahou and Kervaire's work, 
Francisco, H\`a, and Van Tuyl \cite{FHVT2009} introduced the notion
of a {\it Betti splitting} of $I$, a splitting $I = J+K$
so that the graded Betti numbers of
$I$ are a function of the graded Betti numbers
of $J$, $K$, and $J \cap K$.  Betti splittings have
proven to be a useful tool, as is evident from the recent papers
\cite{BH2023,CJRS2025,FHH2023}. While the paper of Francisco, H\`a and Van Tuyl \cite{FHVT2009} considered only the case that $I$ is a monomial ideal, a (Betti)
splitting does not require $I$ to be a monomial ideal.
Consequently, there has been interest in 
splittings of various
families of ideals,
such as ideals of points in $\mathbb{P}^n$ \cite{F2001,F2005,V2005} and
toric ideals \cite{FHKV2021,FKVT2020,GS2025,KT2025a,KT2025b}, as well as interest in variations of splittings, such as partial Betti splittings, which were introduced 
 to study binomial edge ideals in \cite{JSVT2025}.

The purpose of this note is to consider 
splittings of $I_\mathbb{X}$, 
the defining ideal of a set 
of points $\mathbb{X}$ in $\mathbb{P}^1 \times 
\mathbb{P}^1$. 
Ideals of sets of points in $\mathbb{P}^1 \times 
\mathbb{P}^1$ and their quotient rings, while seemingly specialized, provide a rich example of a non-trivial
family of bigraded ideals that can be studied
algebraically, geometrically, and combinatorially (for a sample of some of the literature on such ideals, see
\cite{FGM2018,GMR1992, GKLL2021,GVT2004,GVT2015,HNVT2022,M2024}).
While there has been some work done on the splittings of
ideals of sets of points in $\mathbb{P}^n$, no previous work has considered the case
of points in a multi-projective space.

If $\mathbb{X}$ is a set of points in $\mathbb{P}^{1}\times\mathbb{P}^{1}$ and $\mathcal{G}(I_{\mathbb{X}})=\{g_{1},\dots ,g_{t}\}$ is a minimal generating set of $I_{\mathbb{X}}$, a natural first question to ask is if there
is a partition of $\mathcal{G}(I_{\mathbb{X}})$ into $A\subseteq\mathcal{G}(I_{\mathbb{X}})$ and $\mathcal{G}(I_{\mathbb{X}})\backslash A$ such that 
$I_{\mathbb{X}} = I_{\mathbb{X}_1} + I_{\mathbb{X}_2}$
where $I_{\mathbb{X}_1}=\left<A\right>$, $I_{\mathbb{X}_2}=\left<\mathcal{G}(I_{\mathbb{X}})\backslash A\right>$, and $\mathbb{X}_1$ and $\mathbb{X}_2$ are sets of
points in $\mathbb{P}^1 \times \mathbb{P}^1$?
As we show in Lemma \ref{no_split}, 
the answer to this question is {\it always} no. 
At best, we are only able to
split the ideal as  $I_\mathbb{X} = I_{\mathbb{X}_1} + K$
where $\mathbb{X}_1$ is a set of points and $K$
is a bihomogeneous ideal.  We call such a splitting
a {\it point splitting}.  As shown
in Corollary \ref{cor:pointsplitting}, we prove
that every ideal $I_\mathbb{X}$ has at least one
point splitting.

The existence of a point splitting leads to some
natural follow-up questions: (1) How many point splittings can the ideal $I_\mathbb{X}$ have? and 
(2) Given a point splitting $I_\mathbb{X} = 
I_{\mathbb{X}_1}+K$, what can we say about the ideal $K$?  Under that additional
assumption that $\mathbb{X}$ is arithmetically
Cohen--Macaulay (ACM), that is, $R/I_{\mathbb{X}}$ is
a Cohen--Macaulay ring, we are able to answer 
these questions.  Theorem \ref{pointsplittheorem} 
describes how to find all the point splittings
of $I_{\mathbb{X}}$, thus answering (1).

In order to attempt answering (2) we introduce the notion of unions of lines and ACM points. Using this (slightly) more general framework we are able to obtain Corollary \ref{cor:idealK} which shows that the
ideal $K$ is the defining ideal of a union of
lines and ACM points in $\mathbb{P}^1 \times \mathbb{P}^1$ if $\mathbb{X}$ is an ACM set of points.

In this more general framework of unions of lines and ACM points we manage to obtain the desired results pertaining to Betti splittings and characterize exactly when they occur in Theorem \ref{splitting}, Theorem \ref{cuttheorem} and Corollary \ref{characterization_cor}.
One interesting result of our characterization is Corollary \ref{no_betti_point_splittings} which states that {\it no} point splitting is a Betti splitting. 

Our paper is structured as follows: In Section 
\ref{sec:background}
we recall our relevant background on ideals of sets of points in $\mathbb{P}^1 \times \mathbb{P}^1$. In Section  \ref{sec:splittings} we introduce the notion of a point splitting of $I_\mathbb{X}$, show that every ideal has at least one point splitting, and consider point splittings of $I_{\mathbb{X}}$ under the additional assumption that $\mathbb{X}$ is ACM. In Section \ref{sec:ACM} we introduce the useful framework of ideal of unions of lines and ACM points in $\mathbb{P}^{1}\times\mathbb{P}^{1}$ and apply it to showing that all splittings of ideals of unions of lines and ACM points $I=J+K$ result in $J$, $K$, and $J\cap K$ being ideals of unions of lines and ACM points. In the final section, Section \ref{sec:Betti}, we consider consequences for the graded Betti numbers of $I_{\mathbb{W}}$ where $\mathbb{W}$ is a union of lines and ACM points including the hoped for the Betti splittings.

%%%%%%%%%%%%%%%%%%%%%%%%%%%%%%%%%%%%%%%%%%%%%%%%%%%%%%%%
\section{Background}\label{sec:background}

We start by recalling the relevant background on points in
$\mathbb{P}^1 \times \mathbb{P}^1$. Let 
$R=\mathbb{C}[x_{0},x_{1},y_{0},y_{1}]$ be 
the bigraded ring  where $\deg(x_{0})=\deg(x_{1})=(1,0)$ and $\deg(y_{0})=\deg(y_{1})=(0,1)$. The ring $R$ is
the bigraded coordinate ring  of $\mathbb{P}^1 \times \mathbb{P}^1$. We follow the notation of \cite{GVT2015}; see this reference
for further details.

A point $P \in \mathbb{P}^1 \times \mathbb{P}^1$ has the
form $P = A \times B \in \mathbb{P}^1 \times \mathbb{P}^1$.
If $A=[a_0:a_1]$ and $B = [b_0:b_1]$, then the  defining ideal
of $P$ is the bihomogeneous ideal $I_P = 
\langle a_1x_0-a_0x_1,b_1y_0-b_0y_1\rangle$.
We abuse notation and adopt the convention that $H_{A}$, and respectively $V_{B}$, denotes both the horizontal ruling, respectively, the vertical ruling, and the degree $(1,0)$ form $a_1x_0-a_0x_1$, respectively the degree $(0,1)$ form $b_1y_0-b_0y_1$.
Thus given a point $P=A\times B\in\mathbb{P}^{1}\times\mathbb{P}^{1}$, its defining ideal is given by $I_P=\langle H_{A},V_{B} \rangle$ and geometrically the point, P, is given as the intersection of two lines:
$P=H_{A}\cap V_B$.   Furthermore, we will call $H$ a {\it horizontal
ruling} if $H$ is given by the zero set of a degree $(1,0)$ form,
and $V$ is a {\it vertical ruling} if $V$ is given by the
zero set of a degree $(0,1)$ form.  We sometimes refer to $H$, respectively $V$,
as a horizontal line, respectively, a vertical
line.

Let $\mathbb{X} = \{P_1,\ldots,P_s\} \subseteq \mathbb{P}^1 \times 
\mathbb{P}^1$ be a set of points.  The defining ideal
of $\mathbb{X}$ is $I_{\mathbb{X}} = \bigcap_{i=1}^s I_{P_i}$.  
Let $\pi_{1}:\mathbb{P}^{1}\times\mathbb{P}^{1}\rightarrow\mathbb{P}^{1}$ be the  projection map onto the first coordinate and let $\pi_{2}:\mathbb{P}^{1}\times\mathbb{P}^{1}\rightarrow\mathbb{P}^{1}$ be the projection map onto the second coordinate.

\begin{definition}\label{alphabeta}
    Let $\mathbb{X} \subseteq\mathbb{P}^{1}\times\mathbb{P}^{1}$ be a finite set of reduced points and suppose that $\pi_{1}(\mathbb{X})=\{A_{1},\dots ,A_{h}\}$ and 
    $\pi_{2}(\mathbb{X})=\{B_{1},\dots ,B_{v}\}$. For $i=1,\dots ,h$, set $\alpha_{i}:=|\pi_{1}^{-1}(A_{i})\cap \mathbb{X}|$, and let $\alpha_{\mathbb{X}}:=(\alpha_{1},\dots ,\alpha_{h})$. Similarly, for $j=1,\dots ,v$, set $\beta_{j}:=|\pi_{2}^{-1}(B_{j})\cap \mathbb{X}|$ and $\beta_{\mathbb{X}}:=(\beta_{1},\dots ,\beta_{v})$.
\end{definition}

In the previous definition, $\alpha_i$ counts the number
of points in $\mathbb{X}$ whose first coordinate is $A_i$ and lies 
on the horizontal ruling $H_{A_i}$.  Similarly,
$\beta_j$ is the number of points whose second coordinate
lies on the vertical ruling $V_{B_j}$.  Moving forward,
we make use of the following convention.

\begin{convention}\label{conv:relabel}
    By relabeling the horizontal and vertical rulings, we can always assume that $|\mathbb{X}\cap H_{A_{1}}|
    \ge |\mathbb{X}\cap H_{A_{2}}|\ge\cdots$ and $
    |\mathbb{X}\cap V_{B_{1}}|\ge 
    |\mathbb{X}\cap V_{B_{2}}|\ge\cdots$. That is, we can assume that the first ruling contains the greatest number of points, the second ruling the same number or less, and so on.
\end{convention}

We will primarily be interested in the following class of points:
\begin{definition}
    A set of points $\mathbb{X} \subseteq \mathbb{P}^1 \times \mathbb{P}^1$
    is {\it arithmetically Cohen-Macualay} (ACM) if
    $R/I_{\mathbb{X}}$ is a Cohen-Macualay ring. 
\end{definition}
We can determine if $\mathbb{X}$ is ACM from the 
tuples $\alpha_\mathbb{X}$ and $\beta_\mathbb{X}$.  Given
$\alpha_\mathbb{X} = (\alpha_1,\ldots,\alpha_h)$, we define the 
{\it conjugate} of $\alpha_\mathbb{X}$ to be the tuple
$\alpha_\mathbb{X}^\star = (\alpha_1^\star,\ldots,\alpha_{\alpha_1}^\star)$ where
$\alpha_j^\star = |\{\alpha_i \in \alpha : \alpha_i \geq j\}|$.  We then have the following classification:

\begin{theorem}{\cite[Theorem 4.11]{GVT2015}}\label{thm.acmclassify}
Let $\mathbb{X}$ be a set of points in $\mathbb{P}^1 \times 
\mathbb{P}^1$.  Then $\mathbb{X}$ is ACM if and only if
$\alpha_\mathbb{X}^\star = \beta_\mathbb{X}$.
\end{theorem}
\noindent

\begin{remark}\label{ferrersdiagram}
Informally the above result implies that if the set of points resembles a Ferrer's diagram, then the set of points is ACM.
In particular, given any partition $\alpha = (\alpha_1,\ldots,\alpha_h)$,  $h$ distinct points $A_1,\ldots,A_{h}$
points in $\mathbb{P}^1$, and $\alpha_1$ distinct points
$B_1,\ldots,B_{\alpha_1}$ in $\mathbb{P}^1$, the set
of points \[\{A_i \times B_j : 1 \leq i \leq h, ~1 \leq j \leq \alpha_i\}\] is an ACM set of points in $\mathbb{P}^1 \times \mathbb{P}^1$.
\end{remark}

If $\mathbb{X}$ is ACM, we can describe both the bigraded minimal
free resolution of $I_{\mathbb{X}}$ and the minimal
set of generators using the tuple $\alpha_\mathbb{X}$. 
Note that because of Convention \ref{conv:relabel}, we can
always assume that the entries of $\alpha_\mathbb{X}$ 
and $\beta_\mathbb{X}$ are 
non-increasing.  In fact, both 
$\alpha_\mathbb{X}$ and $\beta_\mathbb{X}$ are integer
partitions of $|\mathbb{X}|$.  We first require the following
definitions:

\begin{definition}\label{bigraded_res_def}
    For any bihomogenous ideal $I\subset R$ there is an associated {\it bigraded minimal free resolution}, that is, an exact sequence of bigraded homomorphsisms of the form
    \[0\rightarrow\bigoplus_{\delta\in\mathbb{N}^{2}}R(-\delta)^{\beta_{i,\delta}(I)}\rightarrow \cdots\rightarrow\bigoplus_{\delta\in\mathbb{N}^{2}}R(-\delta)^{\beta_{1,\delta}(I)}\rightarrow\bigoplus_{\delta\in\mathbb{N}^{2}}R(-\delta)^{\beta_{0,\delta}(I)}\rightarrow I\rightarrow 0\]
    where the summands $R(-\delta)$ are mapped to distinct minimal generators of the preceding module (since $R$ is Noetherian each module is inductively finitely generated so the minimal number of generators is well defined). The $(-\delta)$ represents a shift in the bigrading by $\delta$, that is, we take $1\in R(-\delta)$ to have bidegree $\delta$. The $\beta_{i,\delta}(I)$ are the {\it bigraded Betti numbers of $I$} where $i$ is the {\it homological degree} and $\delta$ is the {\it bidegree}. 
\end{definition}

\begin{definition}\label{CXVX}
Let $\mathbb{X}$ be a set of points in $\mathbb{P}^{1}\times\mathbb{P}^{1}$ and $\alpha_{\mathbb{X}}$ as in Definition \ref{alphabeta}.  Define
\begin{itemize}
       \item $C_{\mathbb{X}}=\{(h,0),(0,\alpha_{1})\}\cup\{(i-1,\alpha_{i}):\alpha_{i}<\alpha_{i-1}\}$, and 
       \item $V_{\mathbb{X}}=\{(h,\alpha_{h})\}\cup\{(i-1,\alpha_{i-1}):\alpha_{i}<\alpha_{i-1}\}$. 
\end{itemize}
\end{definition}

\begin{theorem}{\cite[Theorem 5.2]{GVT2015}}\label{resolution}
    Suppose that $\mathbb{X}$ is an ACM set of points in $\mathbb{P}^{1}\times\mathbb{P}^{1}$ with $\alpha_{\mathbb{X}}=(\alpha_{1},\dots ,\alpha_{h})$. Let $C_{\mathbb{X}}$ and $V_{\mathbb{X}}$ be as in Definition \ref{CXVX}. Then the bigraded minimal free resolution of $I_{\mathbb{X}}$ has the form 

    \[0\longrightarrow\bigoplus_{(v_{1},v_{2})\in V_{\mathbb{X}}}R(-v_{1},-v_{2})\longrightarrow\bigoplus_{(c_{1},c_{2})\in C_{\mathbb{X}}}R(-c_{1},-c_{2})\longrightarrow I_{\mathbb{X}}\longrightarrow 0.\]
\end{theorem}
\noindent
The minimal free resolution of $I_{\mathbb{X}}$ when $\mathbb{X}$ is ACM
was also described in \cite{GMR1992}, but not using
the tuple $\alpha_\mathbb{X}$.

In the statement
below $\beta_i(I_{\mathbb{X}})$ denotes the total
Betti numbers of $I_{\mathbb{X}}$, that is
$\beta_i(I_{\mathbb{X}}) = \sum_{\delta \in \mathbb{N}^2} \beta_{i,\delta}(I_{\mathbb{X}})$. We have the following 
corollary:

\begin{corollary}\label{b1b2}
    Suppose that $\mathbb{X}$ is an ACM set of points in $\mathbb{P}^{1}\times\mathbb{P}^{1}$ with $\alpha_{\mathbb{X}}=(\alpha_{1},\dots ,\alpha_{h})$. Then the total Betti numbers of $I_{\mathbb{X}}$ satisfy: \[\beta_{0}(I_{\mathbb{X}})=
    \beta_{1}(I_{\mathbb{X}})+1.\]
\end{corollary}

\begin{proof}
    This is immediate from Theorem \ref{resolution} and Definition \ref{CXVX} because
    $\beta_0(I_{\mathbb{X}})=|C_\mathbb{X}|$ and $\beta_1(I_{\mathbb{X}}) = |V_\mathbb{X}|$.
\end{proof}

We also have the following description of the 
generators of $I_{\mathbb{X}}$ when $\mathbb{X}$ is ACM.

\begin{theorem}{\cite[Corollary 5.6]{GVT2004}}\label{gens}
    Suppose that $\mathbb{X}$ is an ACM set of points in $\mathbb{P}^{1}\times\mathbb{P}^{1}$ with $\alpha_{\mathbb{X}}=(\alpha_{1},\dots ,\alpha_{h})$. Let $H_{A_{1}},\dots ,H_{A_{h}}$ denote the horizontal rulings and let $V_{B_{1}},\dots ,V_{B_{v}}$ denote the vertical rulings which minimally contain $\mathbb{X}$. Then a minimal bihomogeneous set of generators of $I_{\mathbb{X}}$ is given by  \[\{H_{A_{1}}\cdots H_{A_{h}},V_{B_{1}}\cdots V_{B_{v}}\}\cup \{H_{A_{1}}\cdots H_{A_{i}}\cdot V_{B_{1}}\cdots V_{B_{\alpha_{i+1}}}:\alpha_{i+1}<\alpha_{i}\}.\]
\end{theorem}

\begin{remark}\label{unique_gens}
    One can see that this set of minimal generators is unique up to choice of coefficients in $\mathbb{C}^{*}$. Theorem \ref{resolution} tells us that the generators have specified bi-degrees. If we wished to obtain a different minimal generator of degree, say, $(r-1,\alpha_{r})$, we would need to construct it as a linear combination of generators with bi-degrees $(i,j)$ where $i\le r-1$ and $j\le\alpha_{r}$. However the only generator that satisfies this property is the one of degree $(r-1,\alpha_{r})$ given by Theorem \ref{gens} and so can differ only by a multiple of an element of $\mathbb{C}^{*}$. 
\end{remark}

\begin{remark}\label{drops}
    In Theorem \ref{gens}, some of the generators 
    of $I_\mathbb{X}$ depend upon the location
    of the ``drops'' in $\alpha_X$.  More precisely,
    suppose $\alpha_\mathbb{X} = (\alpha_1,\ldots,\alpha_h)$.
    We let $\{j_1,\ldots,j_r\} \subseteq \{1,\ldots,h\}$
    denote the indices where $\alpha_{{j_i}+1} < \alpha_{j_i}$.   Each ``drop'' corresponds to 
    a generator of the form $H_{A_1}\cdots H_{A_{j_i}}
    V_{B_1}\cdots V_{B_{\alpha_{j_i+1}}}$.  

    Note that we can rewrite $\alpha_{\mathbb{X}}$
    in terms of these drops.  More formally,
    if $\{j_1,\ldots,j_r\}$ are the location
    of the drops, we have
    $$\alpha_\mathbb{X} = 
    (\underbrace{\alpha_1,\ldots,\alpha_1}_{j_1},
    \underbrace{\alpha_{j_1+1},\ldots,\alpha_{j_1+1}}_{j_2-j_1},\alpha_{j_2+1},\ldots,
    \underbrace{\alpha_{j_r+1},\ldots,\alpha_{j_r+1}}_{h-j_r}).$$
\end{remark}

\begin{example}\label{runex}
Consider the following ACM configuration of 20 points in $\mathbb{P}^{1}\times\mathbb{P}^{1}$ given below:
{\small
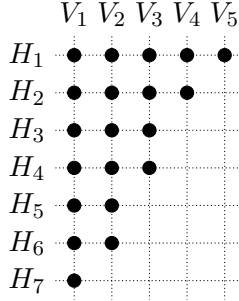
\[
\begin{tikzpicture}[scale =.5]

\draw[fill=black](0,0) circle (5 pt);
\draw[fill=black](1,0) circle (5 pt);
\draw[fill=black](2,0) circle (5 pt);
\draw[fill=black](3,0) circle (5 pt);
\draw[fill=black](4,0) circle (5 pt);
\draw[densely dotted] (-.5,0) -- (4.5,0);

\draw[fill=black](0,-1) circle (5 pt);
\draw[fill=black](1,-1) circle (5 pt);
\draw[fill=black](2,-1) circle (5 pt);
\draw[fill=black](3,-1) circle (5 pt);
\draw[densely dotted] (-.5,-1) -- (4.5,-1);

\draw[fill=black](0,-2) circle (5 pt);
\draw[fill=black](1,-2) circle (5 pt);
\draw[fill=black](2,-2) circle (5 pt);
\draw[densely dotted] (-.5,-2) -- (4.5,-2);

\draw[fill=black](0,-3) circle (5 pt);
\draw[fill=black](1,-3) circle (5 pt);
\draw[fill=black](2,-3) circle (5 pt);
\draw[densely dotted] (-.5,-3) -- (4.5,-3);

\draw[fill=black](0,-4) circle (5 pt);
\draw[fill=black](1,-4) circle (5 pt);
\draw[densely dotted] (-.5,-4) -- (4.5,-4);

\draw[fill=black](0,-5) circle (5 pt);
\draw[fill=black](1,-5) circle (5 pt);
\draw[densely dotted] (-.5,-5) -- (4.5,-5);

\draw[fill=black](0,-6) circle (5 pt);
\draw[densely dotted] (-.5,-6) -- (4.5,-6);

\draw[densely dotted] (0,.5) -- (0,-6.5);
\draw[densely dotted] (1,.5) -- (1,-6.5);
\draw[densely dotted] (2,.5) -- (2,-6.5);
\draw[densely dotted] (3,.5) -- (3,-6.5);
\draw[densely dotted] (4,.5) -- (4,-6.5);

\node at (0,.5) [above]{$V_1$};
\node at (1,.5) [above]{$V_2$};
\node at (2,.5) [above]{$V_3$};
\node at (3,.5) [above]{$V_4$};
\node at (4,.5) [above]{$V_5$};

\node at (-.5,0) [left]{$H_1$};
\node at (-.5,-1) [left]{$H_2$};
\node at (-.5,-2) [left]{$H_3$};
\node at (-.5,-3) [left]{$H_4$};
\node at (-.5,-4) [left]{$H_5$};
\node at (-.5,-5) [left]{$H_6$};
\node at (-.5,-6) [left]{$H_7$};
\end{tikzpicture}
\]}
For this set of points, $\alpha_{\mathbb{X}} = (5,4,3,3,2,2,1)$ and
$\beta_\mathbb{X} = (7,6,4,2,1)$.  A direct computation
will show that $\alpha_\mathbb{X}^\star = 
\beta_\mathbb{X}$, so 
$\mathbb{X}$ is indeed an ACM set of points by Theorem 
\ref{thm.acmclassify}.  

From $\alpha_\mathbb{X}$, we see that $\alpha_2 < \alpha_1$,
$\alpha_3 < \alpha_2$, $\alpha_5 <\alpha_4$, 
and $\alpha_7 < \alpha_6$.  Thus, by Theorem 
\ref{gens}, the generators of $I_\mathbb{X}$ are given
by
\begin{align*}
    I_{\mathbb{X}} =\langle&V_{1}V_{2}V_{3}V_{4}V_{5},
    H_{1}H_{2}H_{3}H_{4}H_{5}H_{6}H_{7},
    H_{1}V_{1}V_{2}V_{3}V_{4}, H_{1}H_{2}V_{1}V_{2}V_{3},\\ &H_{1}H_{2}H_{3}H_{4}V_{1}V_{2},  
    H_{1}H_{2}H_{3}H_{4}H_{5}H_{6}V_{1}\rangle.
\end{align*}
Note that for this $\alpha_\mathbb{X}$, the location
of the drops are $\{1,2,4,6\}$.  These four
drops each correspond to the last four generators 
of $I_\mathbb{X}$ listed above.
\end{example}

While not a direct converse of Theorem
\ref{gens}, if we are given the
generators of an ideal which resemble
the generators of Theorem \ref{gens},
we can prove that the ideal it
defines is an ideal of ACM set of points.

\begin{lemma}\label{gens=>acm}
Let $R = \mathbb{C}[x_0,x_1,y_0,y_1]$.
Fix a positive integer $r > 0$ and
integers $0 < t_1 < t_2 < \cdots < t_r$
and $0 < s_1 < s_2 < \ldots < s_r$.  
Let $H_1,\ldots,H_{t_r}$ be $t_r$ forms of degree $(1,0)$ in $R$ distinct up to multiplication in $\mathbb{C}^{*}$
and let $V_1,\ldots,V_{s_r}$ be $s_r$ forms of degree $(0,1)$ in
$R$ distinct up to multiplication by $\mathbb{C}^{*}$. Consider the ideal $I$ generated by
\footnotesize
\[
\{H_1\cdots H_{t_r},
H_1\cdots H_{t_{r-1}}V_{1}\cdots V_{s_1},
H_1\cdots H_{t_{r-2}}V_1\cdots V_{s_2}, 
\dots, H_1\cdots H_{t_1}V_1\cdots V_{s_{r-1}},
V_1\cdots V_{s_r}\}.
\]
\normalsize
Then $I$ is the defining
ideal of an ACM set of points in 
$\mathbb{P}^1 \times \mathbb{P}^1$.
\end{lemma}

\begin{proof}
    From the $t_i$'s and $s_j$'s, define
    the tuple
    $$\alpha = 
    (\underbrace{s_r,\ldots,s_r}_{t_1},
    \underbrace{s_{r-1},\ldots,s_{r-1}}_{t_2-t_1},\ldots,\underbrace{s_1,\ldots,s_1}_{t_r-t_{r-1}}).$$
    Each $H_i$ defines a distinct point
    $A_i \in \mathbb{P}^1$, and 
    similarly, each $V_j$ defines a distinct
    point $B_j \in \mathbb{P}^1$.  Note that
    we can write $H_{A_i}$, respectively
    $V_{B_j}$, for $H_i$, respectively $V_j$.

    With this notation, we define the
    following set of points in
    $\mathbb{P}^1 \times \mathbb{P}^1$:
    $$\mathbb{X} = 
    \{A_i \times B_j : 1 \leq i \leq t_r, ~~
    1 \leq j \leq \alpha_i\}.$$
    By construction
    (also see Remark \ref{ferrersdiagram}) this is an ACM set
    of points in $\mathbb{P}^1 \times \mathbb{P}^1$ with $\alpha_{\mathbb{X}}  = \alpha$.
    By Theorem \ref{gens}, the generators 
    of this set of points is precisely the
    set of generators given in the statement.
    \end{proof}

    We can extend the previous result by showing that if two
    ideals are built from the same $H_i$'s and $V_j$'s, then
    the intersection of these ideals is also the ideal of an ACM set of
    points.

    \begin{lemma}\label{lem:intersect}
        Let $R = \mathbb{C}[x_0,x_1,y_0,y_1]$.
Fix a positive integer $r > 0$ and
integers 
$$\begin{array}{ll}
0 < t_1 < t_2 < \cdots < t_r,&
0 < t'_1 < \cdots < t'_{r-1} < t_r, \\
0 < s_1 < s_2 < \ldots < s_r, & ~\mbox{and}~~ 0 <s'_1 < \cdots < s'_{r-1} < s_r.
\end{array}
$$
(Note that we are requiring $t'_r = t_r$ and $s'_r = s_r$.)
Let $H_1,\ldots,H_{t_r}$ be $t_r$
forms of degree $(1,0)$ in $R$ distinct up to multiplication by $\mathbb{C}^{*}$
and let $V_1,\ldots,V_{s_r}$ be $s_r$ linear forms of degree $(0,1)$ in
$R$ distinct up to multiplication by $\mathbb{C}^{*}$. Let 
\footnotesize
\begin{eqnarray*}
I & = & \langle H_1\cdots H_{t_r},
H_1\cdots H_{t_{r-1}}V_{1}\cdots V_{s_1},
H_1\cdots H_{t_{r-2}}V_1\cdots V_{s_2}, 
\dots, H_1\cdots H_{t_1}V_1\cdots V_{s_{r-1}},
V_1\cdots V_{s_r}\rangle, ~\mbox{and}~ \\
J & =& \langle H_1\cdots H_{t_r},
H_1\cdots H_{t'_{r-1}}V_{1}\cdots V_{s'_1},
H_1\cdots H_{t'_{r-2}}V_1\cdots V_{s'_2}, 
\dots, H_1\cdots H_{t'_1}V_1\cdots V_{s'_{r-1}},
V_1\cdots V_{s_r}\rangle.
\end{eqnarray*}
\normalsize
Then the ideal $I \cap J$ is the defining
ideal of an ACM set of points in 
$\mathbb{P}^1 \times \mathbb{P}^1$.
    \end{lemma}

\begin{proof}
    By Lemma \ref{gens=>acm}, the ideals $I$ and $J$ define
    sets of ACM points in $\mathbb{P}^1 \times \mathbb{P}^1$,
    say $\mathbb{X}_1$ and $\mathbb{X}_2$, respectively.
    So $I \cap J$ is the defining ideal of the union of these 
    points, i.e., $\mathbb{X}_1 \cup \mathbb{X}_2$.

    Let $A_i$ be the point that $H_i$ defines in $\mathbb{P}^1$
    and $B_j$ be the point that $V_j$ defines in $\mathbb{P}^1$.
    As shown in the proof of Lemma \ref{gens=>acm}, $\mathbb{X}_1$
    is the set 
    $$\mathbb{X}_1 = \{A_i \times B_j : 1 \leq i \leq t_r,~~
    1 \leq j \leq \alpha_i\}$$
    where $$\alpha = 
    (\underbrace{s_r,\ldots,s_r}_{t_1},
    \underbrace{s_{r-1},\ldots,s_{r-1}}_{t_2-t_1},\ldots,\underbrace{s_1,\ldots,s_1}_{t_r-t_{r-1}}).$$
    Similarly, $\mathbb{X}_2$ 
    is the set 
    $$\mathbb{X}_1 = \{A_i \times B_j : 1 \leq i \leq t_r,~~
    1 \leq j \leq \alpha'_i\}$$
    where
    $$\alpha' = 
    (\underbrace{s_r,\ldots,s_r}_{t'_1},
    \underbrace{s'_{r-1},\ldots,s'_{r-1}}_{t'_2-t'_1},\ldots,\underbrace{s'_1,\ldots,s'_1}_{t_r-t'_{r-1}}).$$
    It then follows that set of
    $$\mathbb{X} = \{A_i \times B_j : 1 \leq i \leq t_r, ~~
    1 \leq j \leq \max\{\alpha_i,\alpha'_i\}\},$$
    which can be constructed from the partition $\alpha_{\mathbb{X}}
    = (\max\{\alpha_1,\alpha'_1\},\ldots,
    \max\{\alpha_{t_r},\alpha'_{t_r}\})$.  By 
    Remark \ref{ferrersdiagram} this set of points is an ACM
    set of points in $\mathbb{P}^1 \times \mathbb{P}^1$.
\end{proof}

We require one further lemma for the sequel:

\begin{lemma}\label{lem:intersectpts+line}
    Let $\mathbb{X}$ be an ACM set of points in $\mathbb{P}^1
    \times \mathbb{P}^1$.  For any horizontal ruling $H$,
    we have 
    $$\langle H \rangle \cap I_{\mathbb{X}} = HI_{\mathbb{X} \setminus H}$$
    where $\mathbb{X} \setminus H$ is an ACM set of points.  
    Similarly, for any vertical ruling $V$, 
    we have $\langle V 
    \rangle\cap I_{\mathbb{X}} = V I_{\mathbb{X} \setminus V}$
    where $\mathbb{X} \setminus V$ is an ACM set of points.
\end{lemma}

\begin{proof}
By symmetry, it is enough to prove the statement for
any horizontal ruling.  
By Theorem \ref{gens}, 
 there exists a minimal generating set \[\mathcal{G}(I_{\mathbb{X}})=\{H_{A_{1}}\cdots H_{A_{h}},V_{B_{1}}\cdots V_{B_{v}}\}\cup \{H_{A_{1}}\cdots H_{A_{i}}\cdot V_{B_{1}}\cdots V_{B_{\alpha_{i+1}}}:\alpha_{i+1}<\alpha_{i}\}\] of $I_{\mathbb{X}}$
  where $H_{A_{1}},\dots,H_{A_{h}},V_{B_{1}},\dots,V_{B_{v}}$ are the horizontal and vertical rulings associated with $\mathbb{X}$ and $\alpha_{\mathbb{X}} = (\alpha_1,\ldots,\alpha_h)$ is
  the tuple associated to $\mathbb{X}$.

  For any horizontal ruling $H$, we always have 
  $\langle H \rangle \cap I_{\mathbb{X}} = 
  H(I_{\mathbb{X}}:\langle H \rangle).$  We now consider two cases. For the first case suppose that $H$
  is not among $\{H_{A_{1}},\ldots,H_{A_{h}}\}$, then $I_{\mathbb{X}}:\langle H
  \rangle = I_{\mathbb{X}}$.  Moreover, since $H$ does not
  appear among $\{H_{A_{1}},\ldots,H_{A_{h}}\}$, this means that
  $\mathbb{X} \setminus H = \mathbb{X}$.  So 
  $\langle H \rangle \cap I_{\mathbb{X}} = HI_{\mathbb{X}\setminus H}$
  and $\mathbb{X} \setminus H $ is ACM.

  For the second case, suppose $H = H_{A_{i}} \in \{H_{A_{1}},\ldots,H_{A_{h}}\}$. 
  Then $I_{\mathbb{X}}:\langle H_{A_{i}} \rangle \subseteq I_{\mathbb{X} \setminus H_{A_{i}}}$.  For the reverse inclusion, note that for
  each generator $G \in \mathcal{G}(I_{\mathbb{X}})$,
  we have $G \in I_{\mathbb{X}}:\langle H \rangle$ if $H_{A_{i}}\nmid G$,
  and $G/H_{A_{i}} \in I_{\mathbb{X}}:\langle H \rangle$ if $H_{A_{i}} |G$.  
  But these generators are the generators of the
  ACM set of points constructed from the tuple
  $\alpha = (\alpha_1,\ldots,\hat{\alpha}_i,\ldots,\alpha_h)$,
  i.e., the tuple constructed from $\alpha_{\mathbb{X}}$ by
  removing the $i$-th entry.  But this is the same as removing
  all the points of $\mathbb{X}$ that lie on $H_{A_{i}}$.  
  So $I_{\mathbb{X} \setminus H_{A_{i}}} \subseteq I_{\mathbb{X}}:\langle H_{A_{i}} \rangle$, thus completing the proof.
\end{proof}
%%%%%%%%%%%%%%%%%%%%%%%%%%%%%%%%%%%%%%%%%%%%%%%%%%%%%%
\section{Point splittings}\label{sec:splittings}

As described in the introduction, we wish to create an analogue of splittings for ideals of points in
$\mathbb{P}^1 \times \mathbb{P}^1$. The idea is that we should be able to take a certain class of ideals and partition the minimal generators of these ideals in such a way that the ideals generated by two sets of the partition are in the same class as the original ideal. Further we would like if this partition allows us to compute the (bi)graded Betti numbers of the original ideal by looking only at the ideals generated by the partition and the ideal generated by their intersection. 

Although we will show that by taking a slightly more general class of ideals introduced in Section \ref{sec:ACM} that it is possible to find such splittings, unfortunately, it is not possible for ideals of points in $\mathbb{P}^{1}\times\mathbb{P}^{1}$ as shown by the following result.

\begin{lemma}\label{no_split}
    Let $I_{\mathbb{X}}$ be the defining ideal of a 
    set of points $\mathbb{X}$ in $\mathbb{P}^{1}\times\mathbb{P}^{1}$
    and suppose that $\mathcal{G}(I_{\mathbb{X}})=\{g_{1},\dots ,g_{r}\}$ is a minimal set of generators for $I_{\mathbb{X}}$. Then there is no partition  $\mathcal{G}(I_{\mathbb{X}})=A\sqcup B$ of the minimal generators of $I$, with $A\ne\emptyset$ and $B\ne\emptyset$, such that $J=\langle A\rangle$ and $K=\langle B\rangle$ 
    are both ideals of points in $\mathbb{P}^{1}\times\mathbb{P}^{1}$.
\end{lemma}

\begin{proof}
    It follows from \cite[Theorem 1.2]{GMR1992}, that
    if $I_{\mathbb{X}}$ is the defining ideal of a set of points
    in $\mathbb{P}^1 \times \mathbb{P}^1$, $I_{\mathbb{X}}$ must
    have exactly one minimal generator of bidegree $(a,0)$
    and one minimal generator of bidegree $(0,b)$ for some positive integers $a$ and $b$.  

    Consequently, a minimal set of generators $\mathcal{G}(I_{\mathbb{X}})$ of $I_{\mathbb{X}}$ contains one 
    generator of bidegree $(a,0)$ and one generator of bidegree $(0,b)$, and
    the rest of the generators in $\mathcal{G}(I_{\mathbb{X}})$ have
    bidegree $(i,j)$ with $i > 0$ and $j>0$.  
    For all partitions $\mathcal{G}(I_{\mathbb{X}}) = A \sqcup B$ of the minimal generators of $I_{\mathbb{X}}$, one
    set contains the generator of degree $(a,0)$ and the
    other does not.  Without loss of generality,  say it belongs to $A$.  Thus $B$ has no generator
    of degree $(a,0)$, and thus by \cite[Theorem 1.2]{GMR1992}, the ideal $K$ cannot be the defining ideal
    of a set of points.
    \end{proof}

    By the previous result we know we cannot split the generators such that both ideals are ideals of points.
    The best we can hope for is that one of the two ideals
    is the ideal of a set of points.  In light of
    this idea, we 
    make the following definition.
   
\begin{definition}
    Let $I$ be an ideal of points in $\mathbb{P}^{1}\times\mathbb{P}^{1}$ with minimal generating set $\mathcal{G}(I)$. If there exists a partition $\mathcal{G}(I) = A \sqcup B$ of the minimal generators of $I$ with $J=\langle A\rangle$ and $K=\langle B\rangle$ such that $J$ is also an ideal of points, then we call the partition $I = J+K$   a {\it point splitting} of $I$. 
\end{definition}

\begin{theorem}\label{cor:pointsplitting}
    For every set of points $\mathbb{X} \subseteq \mathbb{P}^{1}\times\mathbb{P}^{1}$ there exists a point splitting of $I_{\mathbb{X}}$.
\end{theorem}

\begin{proof}
For any set of points 
$\mathbb{X} \subset\mathbb{P}^{1}\times\mathbb{P}^{1}$, 
let $H_{A_{1}},\dots ,H_{A_{h}}$ denote the horizontal 
rulings and let $V_{B_{1}},\dots ,V_{B_{v}}$ denote the 
vertical rulings which minimally contain $\mathbb{X}$. 
Then $I_{\mathbb{X}}$ must have the minimal generators 
$H_{A_{1}}\cdots H_{A_{h}}$ and $V_{B_{1}}\cdots 
V_{B_{v}}$ (up to multiplication by elements of $\mathbb{C}^{*}$). 

If $\mathcal{G}(I)$ is a set of minimal generators,
we form the partition $\mathcal{G}(I) = A \sqcup B$
where $A=\{H_{A_{1}}\cdots H_{A_{h}},V_{B_{1}}\cdots V_{B_{v}}\}$.  Then $J=\langle A\rangle$ is an ideal of points 
in $\mathbb{P}^1 \times \mathbb{P}^1$; specifically,
$J$ is the ideal of the points
$$\mathbb{W} = \{A_i \times B_j ~|~ 1 \leq i \leq h, 
1 \leq j \leq v\}.$$
(Notice that in this case $J = I_{\mathbb{W}}$ is the smallest complete intersection which contains the points of $\mathbb{X}$.)
\end{proof}

When $\mathbb{X}$ is an ACM set of points, we have
many possible point splittings of $I_{X}$.

\begin{theorem}\label{pointsplittheorem} 
Let $I_\mathbb{X}$ be an ideal of an ACM set
of points $\mathbb{X}$ in $\mathbb{P}^1
\times \mathbb{P}^1$. Let $H_{A_{1}},\dots ,H_{A_{h}}$ denote the horizontal rulings and let $V_{B_{1}},\dots ,V_{B_{v}}$ denote the vertical rulings which minimally contain $\mathbb{X}$.  
    Suppose that $\mathcal{G}(I_\mathbb{X})=\{f_1,f_2,g_{1},\dots ,g_{r}\}$ is a minimal generating set with $f_{1}=H_{A_{1}}\cdots H_{A_{h}}$ and $f_{2}=V_{B_{1}}\cdots V_{B_{v}}$. 
    For any subset $A' \subseteq \{g_1,\ldots,g_r\}$,
    if  $A = \{f_1,f_2\} \cup A'$ and $B = \mathcal{G}(I) 
    \setminus A$, then $\langle A\rangle + \langle B\rangle$ is  a point
    splitting of $I_{\mathbb{X}}$.
   \end{theorem}

\begin{proof}
    Let $\alpha_{\mathbb{X}} = (\alpha_1,\ldots,
    \alpha_h)$ be the tuple 
    associated to $\mathbb{X}$, and let
    $\pi_1(\mathbb{X}) = \{A_1,\ldots,A_h\}$
    and $\pi_2(\mathbb{X}) = \{B_1,\ldots,B_{\alpha_1}\}$.
    Since $\mathbb{X}$ is ACM, 
    Theorem  \ref{gens} and Remark \ref{unique_gens} implies that 
    the minimal generators of $I_{\mathbb{X}}$ have the form
    \[\{H_{A_{1}}\cdots H_{A_{h}},V_{B_{1}}\cdots V_{B_{v}}\}\cup \{H_{A_1}\cdots H_{A_{i}}V_{B_1}\cdots V_{B_{\alpha_{i+1}}}:\alpha_{i+1}<\alpha_{i}\}
    =\{f_1,f_2\} \cup \{g_1,\ldots,g_r\},\] 
    up to multiplication by an element of $\mathbb{C}^{*}$.
    Using the language of Remark \ref{drops}, there is 
    a one-to-one correspondence between the drops
    $\{j_1,\ldots,j_r\}$
    of  $\alpha_{\mathbb{X}}$ and the  
    generators in $\{g_1,\ldots,g_r\}$.

    Suppose $A' = \{g_{i_1},\ldots,g_{i_t}\}$ and 
    $\{j_{i_1},\ldots,j_{i_t}\}$ are the drops that correspond
    to these generators. Then the elements
    of $A$ can be written
    as 
    \begin{multline}
    \{H_{A_1}\cdots H_{A_r}, 
    H_{A_1}\cdots H_{A_{j_{i_1}}}V_{B_1}\cdots V_{B_{\alpha_{j_{i_1}+1}}},
    H_{A_1}\cdots H_{A_{j_{i_2}}}V_{B_1}\cdots V_{B_
    {\alpha_{j_{i_2}+1}}},
    \ldots, \\
    H_{A_1}\cdots H_{A_{j_{i_t}}}V_{B_1}\cdots V_{B_
    {\alpha_{j_{i_t}+1}}},
    V_{B_1}\cdots V_{B_v}\}.
    \end{multline}
    Now by applying Theorem \ref{gens=>acm}, we can conclude
    that the ideal generated by these elements 
    defines an ACM set of points in $\mathbb{P}^1 
    \times \mathbb{P}^1$.  Thus $\langle A \rangle + \langle B \rangle$
    is a point splitting of $I_{\mathbb{X}}$.
\end{proof}

\begin{example}
To illustrate the necessity of the ACM condition for 
the previous theorem we examine the following example 
of an ideal of points which violates the theorem stated without the ACM condition.
Let $\mathbb{X}$ be the following set of six points.
\[
\begin{array}{lll}
P_{1}=[1:1]\times[1:1] & P_{2}=[1:1]\times[2:1] & P_{3}=[1:1]\times[3:1] \\ P_{4}=[2:1]\times[1:1] & P_{5}=[2:1]\times[2:1] & P_{6}=[3:1]\times[4:1]. 
\end{array}
\]

These points can be visualized as 
{\small
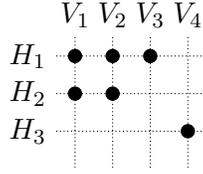
\[
\begin{tikzpicture}[scale =.5]

\draw[fill=black](0,0) circle (5 pt);
\draw[fill=black](1,0) circle (5 pt);
\draw[fill=black](2,0) circle (5 pt);

\draw[densely dotted] (-.5,0) -- (3.5,0);

\draw[fill=black](0,-1) circle (5 pt);
\draw[fill=black](1,-1) circle (5 pt);

\draw[densely dotted] (-.5,-1) -- (3.5,-1);

\draw[fill=black](3,-2) circle (5 pt);

\draw[densely dotted] (-.5,-2) -- (3.5,-2);

\draw[densely dotted] (0,.5) -- (0,-3.0);
\draw[densely dotted] (1,.5) -- (1,-3.0);
\draw[densely dotted] (2,.5) -- (2,-3.0);
\draw[densely dotted] (3,.5) -- (3,-3.0);

\node at (0,.5) [above]{$V_1$};
\node at (1,.5) [above]{$V_2$};
\node at (2,.5) [above]{$V_3$};
\node at (3,.5) [above]{$V_4$};

\node at (-.5,0) [left]{$H_1$};
\node at (-.5,-1) [left]{$H_2$};
\node at (-.5,-2) [left]{$H_3$};
%\node at (-.5,-3) [left]{$H_4$};

\end{tikzpicture}
\]}
Since $\alpha_{\mathbb{X}}=(3,2,1) $,  $\beta_{\mathbb{X}} = (2,2,1,1)$ and $\beta_{\mathbb{X}}^\star = (4,2)$, this set of points is not ACM.

Using Macaulay2, we can find a list of minimal 
generators:
\begin{align*} I=\langle &f_{1}=(y_{0}-y_{1})(2y_{0}-y_{1})(3y_{0}-y_{1})(4y_{0}-y_{1}), \\
&f_{2}=(x_{0}-x_{1})(2x_{0}-x_{1})(3x_{0}-x_{1}), \\ 
&g_{1}=(x_0-x_1)(2x_{0}-x_{1})(4y_{0}-y_{1}),\\
&g_{2}=(y_{0}-y_{1})(2y_{0}-y_{1})(24x_{1}y_{0}-3x_{0}y_{1}-5x_{1}y_{1}), \\
&g_{3}=(y_{0}-y_{1})(2y_{0}-y_{1})(24x_{0}y_{0}-9x_{0}y_{1}+x_{1}y_{1}),\\
&g_{4}=(x_{0}-x_{1})(16x_{1}y_{0}^{2}-24x_{1}y_{0}y_{1}+18x_{0}y_{1}^{2}-x_{1}y_{1}^{2})\rangle.
\end{align*}
Using the notation of the previous theorem, we 
let $A'=\{g_{2}\}$. 
Then $A=\{f_{1},f_{2},g_{2}\}$ and $B=\{g_{1},g_{3},g_{4}\}$. 
We examine the ideal $J=\langle A\rangle$. To show that $J$ is not an ideal of points and hence not a summand of a point splitting of $I$ we compute the primary decomposition of $J$. The primary ideals that appear in the decomposition
are:
\footnotesize
\begin{align*}
Q_{1}=& \langle x_{0}-x_{1},\,  y_{0}-y_{1} \rangle,~~
Q_{2}= \langle 2\,x_{0}-x_{1},\, y_{0}-y_{1} \rangle,~~ Q_{3}= \langle 3\,x_{0}-x_{1},\, y_{0}-y_{1} \rangle, \\
Q_{4}=& \langle x_{0}-x_{1},\, 2\,y_{0}-y_{1} \rangle,~~ Q_{5}= \langle 2\,x_{0}-x_{1},\, 2\,y_{0}-y_{1} \rangle,~~ Q_{6}= \langle 3\,x_{0}-x_{1},\, 2\,y_{0}-y_{1} \rangle, \\
Q_{7}=& \langle x_{0}-x_{1},\, 3\,y_{0}-y_{1} \rangle, Q_{8}=\langle 3\,x_{0}-x_{1},\, 4\,y_{0}-y_{1} \rangle, \\
Q_{9}=& \langle 2\,x_{0}-x_{1}, \,25\,y_{0}y_{1}^{2}-7\,y_{1}^{3}, \,25\,y_{0}^{2}y_{1}-2\,y_{1}^{3},\,300\,y_{0}^{3}-7\,y_{1}^{3}\rangle,\\
Q_{10}=& \langle x_{1}^{3}, \,x_{0}x_{1}^{2},\,x_{0}^{2}x_{1}, \,x_{0}^{3},\,y_{1}^{4},\,y_{0}y_{1}^{3}, \,y_{0}^{2}y_{1}^{2}, \,y_{0}^{3}y_{1}, \,y_{0}^{4}, \\ 
&48\,x_{1}y_{0}^{3}-6\,x_{0}y_{0}^{2}y_{1}-82\,x_{1}y_{0}^{2}y_{1}+9\,x_{0}y_{0}y_{1}^{2}+39\,x_{1}y_{0}y_{1}^{2}-3\,x_{0}y_{1}^{3}-5\,x_{1}y_{1}^{3}, \\ &11\,x_{0}x_{1}y_{0}y_{1}^{2}+7\,x_{1}^{2}y_{0}y_{1}^{2}-5\,x_{0}x_{1}y_{1}^{3}-x_{1}^{2}y_{1}^{3}, 121\,x_{0}^{2}y_{0}y_{1}^{2}-49\,x_{1}^{2}y_{0}y_{1}^{2}-55\,x_{0}^{2}y_{1}^{3}+24\,x_{0}x_{1}y_{1}^{3}+7\,x_{1}^{2}y_{1}^{3}, \\
&12\,x_{1}^{2}y_{0}^{2}y_{1}-7\,x_{1}^{2}y_{0}y_{1}^{2}+x_{1}^{2}y_{1}^{3},\,132\,x_{0}x_{1}y_{0}^{2}y_{1}+49\,x_{1}^{2}y_{0}y_{1}^{2}-24\,x_{0}x_{1}y_{1}^{3}-7\,x_{1}^{2}y_{1}^{3}, \\
&1452\,x_{0}^{2}y_{0}^{2}y_{1}-343\,x_{1}^{2}y_{0}y_{1}^{2}-264\,x_{0}^{2}y_{1}^{3}+168\,x_{0}x_{1}y_{1}^{3}+49\,x_{1}^{2}y_{1}^{3}, \\
&17424x_{0}^{2}y_{0}^{3}-1813x_{1}^{2}y_{0}y_{1}^{2}-1188x_{0}^{2}y_{1}^{3}+888x_{0}x_{1}y_{1}^{3}+259x_{1}^{2}y_{1}^{3}\rangle.
\end{align*}
\normalsize
From the decomposition, we see that $J$ is not the
ideal of a set of points since $Q_{9}$ and $Q_{10}$ do not correspond to a points. Thus, we do not have a point
splitting. \\

One may then suspect that when $\mathbb{X}$ is not ACM, we cannot obtain a point splitting. However we could have chosen an
alternative partition, namely $A'=\{g_{1}\}$, and so 
$A=\{f_{1},f_{2},g_{1}\}$ and $B=\{g_{2},g_{3},g_{4}\}$. Let $J=\langle A\rangle$. Then examining the primary decomposition we get the associated primes
\[
\begin{array}{lll}
Q_{1}=\langle x_{0}-x_{1},\, y_{0}-y_{1} \rangle & Q_{2}=\langle 2\,x_{0}-x_{1},\, y_{0}-y_{1} \rangle &
Q_{3}=\langle x_{0}-x_{1},\, 2\,y_{0}-y_{1} \rangle \\
Q_{4}=\langle 2\,x_{0}-x_{1},\, 2\,y_{0}-y_{1} \rangle & Q_{5}=\langle x_{0}-x_{1},\, 3\,y_{0}-y_{1} \rangle & Q_{6}=\langle 2\,x_{0}-x_{1},\, 3\,y_{0}-y_{1} \rangle \\
Q_{7}=\langle x_{0}-x_{1},\, 4\,y_{0}-y_{1} \rangle & Q_{8}=\langle 2\,x_{0}-x_{1},\, 4\,y_{0}-y_{1} \rangle  & Q_{9}=\langle 3\,x_{0}-x_{1},\, 4\,y_{0}-y_{1} \rangle.
\end{array}
\]
So, the ideal $J$ corresponds to the set of points
{\small
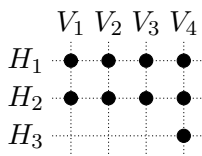
\[
\begin{tikzpicture}[scale =.5]

\draw[fill=black](0,0) circle (5 pt);
\draw[fill=black](1,0) circle (5 pt);
\draw[fill=black](2,0) circle (5 pt);
\draw[fill=black](3,0) circle (5 pt);

\draw[densely dotted] (-.5,0) -- (3.5,0);

\draw[fill=black](0,-1) circle (5 pt);
\draw[fill=black](1,-1) circle (5 pt);
\draw[fill=black](2,-1) circle (5 pt);
\draw[fill=black](3,-1) circle (5 pt);

\draw[densely dotted] (-.5,-1) -- (3.5,-1);

\draw[fill=black](3,-2) circle (5 pt);

\draw[densely dotted] (-.5,-2) -- (3.5,-2);

\draw[densely dotted] (0,.5) -- (0,-2.5);
\draw[densely dotted] (1,.5) -- (1,-2.5);
\draw[densely dotted] (2,.5) -- (2,-2.5);
\draw[densely dotted] (3,.5) -- (3,-2.5);

\node at (-1.2,0) {$H_1$};
\node at (-1.2,-1) {$H_2$};
\node at (-1.2,-2) {$H_3$};

\node at (0,1) {$V_1$};
\node at (1,1) {$V_2$};
\node at (2,1) {$V_3$};
\node at (3,1) {$V_4$};

\end{tikzpicture}
\]}
So, for this partition, we get a point splitting.
\end{example}

\begin{remark}
    In the previous example we managed to obtain a point splitting by taking the partition that contains all generators which are products of linear forms. It seems possible that doing so always results in a point splitting. Perhaps this even characterizes the partitions of $I$ that give point splittings. This question will be explored in future work.  
\end{remark}

%%%%%%%%%%%%%%%%%%%%%%%%%%%%%%%%%%%%%%%%%%%%%%%%%%%%%%%%%%%%%%%%%%%%%%%%%%
\section{Unions of lines and ACM Sets of Points}\label{sec:ACM}

Suppose $I_{\mathbb{X}}$ is an ideal of an ACM
set of points in $\mathbb{P}^{1}\times\mathbb{P}^{1}$ with generating set $\mathcal{G}(I_{\mathbb{X}})$.
By Theorem \ref{pointsplittheorem} we know 
how to get a  partition $\mathcal{G}(I_{\mathbb{X}}) = A \sqcup B$ of $\mathcal{G}(I_{\mathbb{X}})$
so that $\langle A\rangle$ is also an ideal of an ACM set of points.
While an arbitrary partition of $\mathcal{G}(I_{\mathbb{X}})$ may not result 
in an ideal of points, this partition still has
some geometric structure.
In particular, if $I_{\mathbb{X}} = \langle A\rangle + \langle B\rangle$
with $\mathcal{G}(I_{\mathbb{X}}) = A \sqcup B$, then we will show that  both $\langle A\rangle$ and
$\langle B\rangle$ are ideals of  an ACM
sets of points {\it and/or} a collection of lines
in $\mathbb{P}^1 \times \mathbb{P}^1$.
We start with some relevant definitions.

\begin{definition}\label{PandL}
    Let $\mathbb{W}=\{P_{1},\dots ,P_{r},
    H_1,\ldots,H_t,V_1,\ldots,V_p\}  \subset\mathbb{P}^{1}\times\mathbb{P}^{1}$ where the $P_i$'s
    are points, the $H_i$'s are
    horizontal lines of type $(1,0)$, and
    the $V_i$'s are vertical
    lines of type $(0,1)$. Additionally we assume that none of the $P_{i}$ lie on the $H_{j}$ or $V_{k}$.
    Then we call $\mathbb{W}$ a {\it union of lines and points} in $\mathbb{P}^{1}\times\mathbb{P}^{1}$.
     We say that $\mathbb{W}$ is a {\it union of lines and ACM points} in $\mathbb{P}^{1}\times\mathbb{P}^{1}$ if the set of points
     $\mathbb{X}_{\mathbb{W}} = \{P_1,\ldots,P_s\}$ is 
     an ACM set of points.
\end{definition}

\begin{example}
Let $\mathbb{W}$ be the following union of lines and points in $\mathbb{P}^{1}\times\mathbb{P}^{1}$:

$$\mathbb{W}=\{[2:1]\times[3:1],[2:1]\times[4:1],[3:1]\times[3:1],[4:1]\times[3:1], H_1, V_1, V_2\}.$$
where $H_1$ is a horizontal line of bidegree $(1,0)$, $V_1$ and $V_2$ two vertical lines of bidegree $(0,1)$, given by:
\begin{eqnarray*}
H_1  = V(x_{0}-x_{1}), ~~ 
V_1 & =& V(y_{0}-y_{1}), ~~\mbox{and}~~ V_2  =V(y_{0}-2y_{1}).
\end{eqnarray*}
Then $\mathbb{W}$ can be represented 
as
{\small
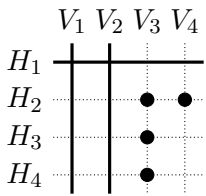
\[
\begin{tikzpicture}[scale =.5]

\draw[very thick] (-.5,0) -- (3.5,0);

\draw[fill=black](2,-1) circle (5 pt);
\draw[fill=black](3,-1) circle (5 pt);
\draw[densely dotted] (-.5,-1) -- (3.5,-1);

\draw[fill=black](2,-2) circle (5 pt);
\draw[densely dotted] (-.5,-2) -- (3.5,-2);

\draw[fill=black](2,-3) circle (5 pt);
\draw[densely dotted] (-.5,-3) -- (3.5,-3);

\draw[very thick] (0,.5) -- (0,-3.5);
\draw[very thick] (1,.5) -- (1,-3.5);
\draw[densely dotted] (2,.5) -- (2,-3.5);
\draw[densely dotted] (3,.5) -- (3,-3.5);

\node at (0,.5) [above]{$V_1$};
\node at (1,.5) [above]{$V_2$};
\node at (2,.5) [above]{$V_3$};
\node at (3,.5) [above]{$V_4$};

\node at (-.5,0) [left]{$H_1$};
\node at (-.5,-1) [left]{$H_2$};
\node at (-.5,-2) [left]{$H_3$};
\node at (-.5,-3) [left]{$H_4$};

\end{tikzpicture}
\]}
In this example, $\mathbb{W}$ is a union of
lines and ACM points since $\mathbb{X}_\mathbb{W}$
is an ACM set of points.
\end{example}

We have the following structure result for
the ideal of a union of lines and points.

\begin{theorem}\label{product}
        Let $\mathbb{W}=\{P_{1},\dots ,P_{r},
    H_1,\ldots,H_t,V_1,\ldots,V_p\} = 
        \mathbb{X}_{\mathbb{W}} \cup \{H_1,\ldots,H_t,V_1,\ldots,V_p\}$ 
        be a union of lines and points 
        in $\mathbb{P}^1 \times \mathbb{P}^1$.
        Assume all points and lines are unique and no points of $\mathbb{X}_{\mathbb{W}}$ lie on any of the lines.
        Then $$I_{\mathbb{W}}=
        I_{\mathbb{X}_{\mathbb{W}}}
        \left(\prod_{i=1}^{t}H_{i}\right)
        \left(\prod_{j=1}^{p}V_{j}\right)
        $$
        where we  abuse notation to let
        $H_i$ and $V_{j}$ denote both the line and the form
        defining the line.
\end{theorem}

\begin{proof}
By definition
    \[
    I_{\mathbb{W}}=
    \left(\bigcap_{i=1}^{r}I_{P_{i}}\right)\cap
    \left(\bigcap_{i=1}^{t}I(H_{i})\right) \cap
    \left(\bigcap_{j=1}^{p}I(V_{j})\right)
    =
    I_{\mathbb{X}_{\mathbb{W}}} \cap
    \left(\bigcap_{i=1}^{t}I(H_{i})\right) \cap
    \left(\bigcap_{j=1}^{p}I(V_{j})\right)
    .\]
    Now we note that
    each $I(H_{i})$ is generated by a single 
    irreducible element $a_{1,i}x_{0}-a_{0,i}x_{1}$
    and each $I(V_{j})$ is generated by a single irreducible element $b_{1,j}y_{0}-b_{0,j}y_{1}$.
    Thus we have 
    \[
    \left(\bigcap_{i=1}^{t}I(H_{i})\right) \cap
    \left(\bigcap_{j=1}^{p}I(V_{j})\right)
    =
     \left(\prod_{i=1}^{t}H_{i}\right)
        \left(\prod_{j=1}^{p}V_{j}\right)
        .\] 

Because $A \cap B \supseteq AB$ for any two ideals $A$
and $B$, we have 
the $\supseteq$ containment.
For the reverse containment,
since $
\left(\prod_{i=1}^{t}H_{i}\right)\left(\prod_{j=1}^{p}V_{j}\right)
$ is principal, any element that belongs to $I_{\mathbb{W}}$ is a multiple of the generator. Further, by assumption, we have that none of the points
of $\mathbb{X}_\mathbb{W}$ lie on any of the $H_{i}$ or $V_{j}$. Therefore 
for any 
$$f = g\cdot \left(\prod_{i=1}^{t}H_{i}\right)
\left(\prod_{j=1}^{p}V_{j}\right)
\in I_{\mathbb{W}} = 
 I_{\mathbb{X}_{\mathbb{W}}} \cap
    \left(\bigcap_{i=1}^{t}I(H_{i})\right) \cap
    \left(\bigcap_{j=1}^{p}I(V_{j})\right)
,$$ 
if 
$P\in \mathbb{X}_{\mathbb{W}}$, 
then $f(P) =g\cdot
\left(\prod_{i=1}^{t}H_{i}\right)\left(\prod_{j=1}^{p}V_{j}\right)
(P)=0$ implies that $g\in I_{\mathbb{X}_{\mathbb{W}}}$.
Hence \[I_{\mathbb{W}}=\left(\prod_{i=1}^{t}H_{i}\right)
        \left(\prod_{j=1}^{p}V_{j}\right)
I_{\mathbb{X}_{\mathbb{W}}}\] as required. 
\end{proof}

\begin{corollary}\label{isACM}
 Let  $\mathbb{W}=\{P_{1},\dots ,P_{r},
    H_1,\ldots,H_t,V_1,\ldots,V_p\} = 
        \mathbb{X}_{\mathbb{W}} \cup \{H_1,\ldots,H_t,V_1,\ldots,V_p\}$ 
        be a union of points and lines
        in $\mathbb{P}^1 \times \mathbb{P}^1$.
        Assume all points and lines are unique and no points of $\mathbb{X}_{\mathbb{W}}$ lie on any of the lines.
    Then the resolution of $I_{\mathbb{W}}$ is identical to the resolution of $I_{\mathbb{X}_{\mathbb{W}}}$ with every generator multiplied by the product of the lines and with the grading shifted by the degree of this product.
\end{corollary}

\begin{proof}
    Since $I_{\mathbb{W}}$ is just a multiple of $I_{{\mathbb{X}_{\mathbb{W}}}}$, the minimal resolution of
    the two ideals will be exactly the same just with a shift in grading corresponding to the degree of $\left(\prod_{i=1}^{t}H_{i}\right)
        \left(\prod_{j=1}^{p}V_{j}\right)$.
\end{proof}

\begin{theorem}\label{subsetisACM}
        Let $\mathbb{W}$
        be a union of lines and ACM points in $\mathbb{P}^{1}\times\mathbb{P}^{1}$,
        and let
        $\mathcal{G}(I_{\mathbb{W}})=\{g_{1},\ldots,g_{r}\}$ be a minimal set of generators for $I_{\mathbb{W}}$. Then for any non-empty subset $S \subseteq \mathcal{G}(I_{\mathbb{W}})$,
        the ideal $\langle S\rangle$ is the defining ideal
        of a union of lines and ACM points 
        (we allow the case of all lines or all points).
\end{theorem}

\begin{proof}
    Suppose $\mathbb{W}=\{P_{1},\dots ,P_{r}, H_1,\ldots,H_t,V_1,\ldots,V_p\} = \mathbb{X}_{\mathbb{W}} \cup \{H_1,\ldots,H_t,V_1,\ldots,V_p\}$ is a union of lines and  an ACM set of points in $\mathbb{P}^{1}\times\mathbb{P}^{1}$. Then by Theorem \ref{product} we have $I_{\mathbb{W}}=
        I_{\mathbb{X}_{\mathbb{W}}}
        \left(\prod_{i=1}^{t}H_{i}\right)
        \left(\prod_{j=1}^{p}V_{j}\right)
        $. So it suffices to consider the case where 
    $\mathbb{W}=\mathbb{X}_{\mathbb{W}}$ since we can just factor out the product of the lines. Therefore we consider the case where 
    $\mathbb{W}=\mathbb{X}_{\mathbb{W}}$ is an ACM set of points.

    By Theorem \ref{gens} we have a minimal homogeneous set of generators of $I_{\mathbb{W}}$ given by  
    \[\mathcal{G}(I_{\mathbb{W}})=\{H_{A_{1}}\cdots H_{A_{h}},
    V_{B_{1}}\cdots V_{B_{v}}\}\cup 
    \{H_{A_{1}}\cdots H_{A_{i}}\cdot V_{B_{1}}\cdots V_{B_{\alpha_{i+1}}}:\alpha_{i+1}<\alpha_{i}\}\] where $H_{A_{1}},\dots,H_{A_{h}},V_{B_{1}},\dots,V_{B_{v}}$ are the horizontal and vertical rulings associated with $\mathbb{W}=\mathbb{X}_{\mathbb{W}}$ and $\alpha_{\mathbb{X}_{\mathbb{W}}} = (\alpha_1,\ldots,\alpha_h)$ is
  the tuple associated to $\mathbb{X}_{\mathbb{W}}$. Suppose that $S$ is a subset of these generators. If $S$ contains both $H_{A_{1}}\cdots H_{A_{h}}$ and $V_{B_{1}}\cdots V_{B_{v}}$, then we are in the situation of Theorem \ref{pointsplittheorem} and $S$ generates an ACM set of points.

    Suppose that $S$ contains at most one
    of $H_{A_{1}}\cdots H_{A_{h}}$ or $V_{B_{1}}\cdots V_{B_{v}}$.
    Factor out the greatest common divisor $D$ of 
    the elements of $S$.  Note that this element will have
    the form $H_{A_1}\cdots H_{A_a}V_{B_1}\cdots V_{B_b}$ 
    for some $a$ and $b$. (It might be the case that $a=0$ or $b=0$,
    that is, no $H_{A_i}$ or $V_{B_j}$ divides all the elements of
    $S$.)  The set of polynomials 
    in the set $S'=\{F/D : F \in S\}$ will then have
    the form of the polynomials that appear in Theorem \ref{gens=>acm}.
    Thus, the elements in this set generates an
    ACM sets of points in $\mathbb{P}^1 \times \mathbb{P}^1$, say
    $\mathbb{Y}$.  Thus 
    $$\langle S \rangle = D\langle S' \rangle = D I_{\mathbb{Y}}.$$
    In other words, $S$ generates the ideal
    of a union of lines and ACM set of points.
\end{proof}

\begin{example} Continuing with the example from Example \ref{runex}, we have the following ACM configuration of points in $\mathbb{P}^{1}\times\mathbb{P}^{1}$:

{\small
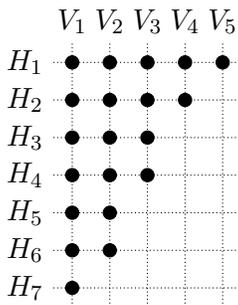
\[
\begin{tikzpicture}[scale =.5]

\draw[fill=black](0,0) circle (5 pt);
\draw[fill=black](1,0) circle (5 pt);
\draw[fill=black](2,0) circle (5 pt);
\draw[fill=black](3,0) circle (5 pt);
\draw[fill=black](4,0) circle (5 pt);
\draw[densely dotted] (-.5,0) -- (4.5,0);

\draw[fill=black](0,-1) circle (5 pt);
\draw[fill=black](1,-1) circle (5 pt);
\draw[fill=black](2,-1) circle (5 pt);
\draw[fill=black](3,-1) circle (5 pt);
\draw[densely dotted] (-.5,-1) -- (4.5,-1);

\draw[fill=black](0,-2) circle (5 pt);
\draw[fill=black](1,-2) circle (5 pt);
\draw[fill=black](2,-2) circle (5 pt);
\draw[densely dotted] (-.5,-2) -- (4.5,-2);

\draw[fill=black](0,-3) circle (5 pt);
\draw[fill=black](1,-3) circle (5 pt);
\draw[fill=black](2,-3) circle (5 pt);
\draw[densely dotted] (-.5,-3) -- (4.5,-3);

\draw[fill=black](0,-4) circle (5 pt);
\draw[fill=black](1,-4) circle (5 pt);
\draw[densely dotted] (-.5,-4) -- (4.5,-4);

\draw[fill=black](0,-5) circle (5 pt);
\draw[fill=black](1,-5) circle (5 pt);
\draw[densely dotted] (-.5,-5) -- (4.5,-5);

\draw[fill=black](0,-6) circle (5 pt);
\draw[densely dotted] (-.5,-6) -- (4.5,-6);

\draw[densely dotted] (0,.5) -- (0,-6.5);
\draw[densely dotted] (1,.5) -- (1,-6.5);
\draw[densely dotted] (2,.5) -- (2,-6.5);
\draw[densely dotted] (3,.5) -- (3,-6.5);
\draw[densely dotted] (4,.5) -- (4,-6.5);

\node at (0,.5) [above]{$V_1$};
\node at (1,.5) [above]{$V_2$};
\node at (2,.5) [above]{$V_3$};
\node at (3,.5) [above]{$V_4$};
\node at (4,.5) [above]{$V_5$};

\node at (-.5,0) [left]{$H_1$};
\node at (-.5,-1) [left]{$H_2$};
\node at (-.5,-2) [left]{$H_3$};
\node at (-.5,-3) [left]{$H_4$};
\node at (-.5,-4) [left]{$H_5$};
\node at (-.5,-5) [left]{$H_6$};
\node at (-.5,-6) [left]{$H_7$};
\end{tikzpicture}
\]}

with minimal generators given by 
\begin{align*}
    \mathcal{G}(I_{\mathbb{X}}) =\{&V_{1}V_{2}V_{3}V_{4}V_{5}, H_{1}V_{1}V_{2}V_{3}V_{4}, H_{1}H_{2}V_{1}V_{2}V_{3}, H_{1}H_{2}H_{3}H_{4}V_{1}V_{2}, \\ 
    &H_{1}H_{2}H_{3}H_{4}H_{5}H_{6}V_{1}, H_{1}H_{2}H_{3}H_{4}H_{5}H_{6}H_{7}\}.
\end{align*}

We will take the following subset $S\subseteq\mathcal{G}(I_{\mathbb{X}})$: \[S=\{V_{1}V_{2}V_{3}V_{4}V_{5},H_{1}H_{2}V_{1}V_{2}V_{3},H_{1}H_{2}H_{3}H_{4}H_{5}H_{6}V_{1}\}.\] Following the previous proof we have that $D=V_{1}$, and so \[S'=\{F/D:F\in S\}=\{V_{2}V_{3}V_{4}V_{5},H_{1}H_{2}V_{2}V_{3},H_{1}H_{2}H_{3}H_{4}H_{5}H_{6}\}.\] By Lemma \ref{gens=>acm}, this corresponds to the following configuration of ACM points $\mathbb{Y}$:

{\small
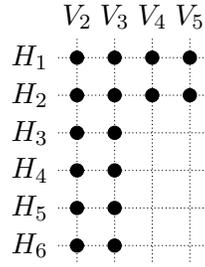
\[
\begin{tikzpicture}[scale =.5]

\draw[fill=black](1,0) circle (5 pt);
\draw[fill=black](2,0) circle (5 pt);
\draw[fill=black](3,0) circle (5 pt);
\draw[fill=black](4,0) circle (5 pt);

\draw[fill=black](1,-1) circle (5 pt);
\draw[fill=black](2,-1) circle (5 pt);
\draw[fill=black](3,-1) circle (5 pt);
\draw[fill=black](4,-1) circle (5 pt);

\draw[fill=black](1,-2) circle (5 pt);
\draw[fill=black](2,-2) circle (5 pt);

\draw[fill=black](1,-3) circle (5 pt);
\draw[fill=black](2,-3) circle (5 pt);

\draw[fill=black](1,-4) circle (5 pt);
\draw[fill=black](2,-4) circle (5 pt);

\draw[fill=black](1,-5) circle (5 pt);
\draw[fill=black](2,-5) circle (5 pt);

\draw[densely dotted] (0.5,0) -- (4.5,0);
\draw[densely dotted] (0.5,-1) -- (4.5,-1);
\draw[densely dotted] (0.5,-2) -- (4.5,-2);
\draw[densely dotted] (0.5,-3) -- (4.5,-3);
\draw[densely dotted] (0.5,-4) -- (4.5,-4);
\draw[densely dotted] (0.5,-5) -- (4.5,-5);

\draw[densely dotted] (1,0.5) -- (1,-5.5);
\draw[densely dotted] (2,0.5) -- (2,-5.5);
\draw[densely dotted] (3,0.5) -- (3,-5.5);
\draw[densely dotted] (4,0.5) -- (4,-5.5);

\node at (1,.5) [above]{$V_2$};
\node at (2,.5) [above]{$V_3$};
\node at (3,.5) [above]{$V_4$};
\node at (4,.5) [above]{$V_5$};

\node at (0.5,0) [left]{$H_1$};
\node at (0.5,-1) [left]{$H_2$};
\node at (0.5,-2) [left]{$H_3$};
\node at (0.5,-3) [left]{$H_4$};
\node at (0.5,-4) [left]{$H_5$};
\node at (0.5,-5) [left]{$H_6$};

\end{tikzpicture}
\]}
Then multiplying $I_{\mathbb{Y}}$ by $D$ yields: 
\begin{align*}
DI_{\mathbb{Y}}&=V_{1}\langle\{V_{2}V_{3}V_{4}V_{5},H_{1}H_{2}V_{2}V_{3},H_{1}H_{2}H_{3}H_{4}H_{5}H_{6}\}\rangle \\
&=\langle\{V_{1}V_{2}V_{3}V_{4}V_{5},H_{1}H_{2}V_{1}V_{2}V_{3},H_{1}H_{2}H_{3}H_{4}H_{5}H_{6}V_{1}\}\rangle \\
&=\langle S\rangle
\end{align*}
The ideal $\left<S\right>$ corresponds to the following union of lines and ACM points 
{\small\[
\begin{tikzpicture}[scale =.5]

\draw[fill=black](1,0) circle (5 pt);
\draw[fill=black](2,0) circle (5 pt);
\draw[fill=black](3,0) circle (5 pt);
\draw[fill=black](4,0) circle (5 pt);

\draw[fill=black](1,-1) circle (5 pt);
\draw[fill=black](2,-1) circle (5 pt);
\draw[fill=black](3,-1) circle (5 pt);
\draw[fill=black](4,-1) circle (5 pt);

\draw[fill=black](1,-2) circle (5 pt);
\draw[fill=black](2,-2) circle (5 pt);

\draw[fill=black](1,-3) circle (5 pt);
\draw[fill=black](2,-3) circle (5 pt);

\draw[fill=black](1,-4) circle (5 pt);
\draw[fill=black](2,-4) circle (5 pt);

\draw[fill=black](1,-5) circle (5 pt);
\draw[fill=black](2,-5) circle (5 pt);

\draw[densely dotted] (-.5,0) -- (4.5,0);
\draw[densely dotted] (-.5,-1) -- (4.5,-1);
\draw[densely dotted] (-.5,-2) -- (4.5,-2);
\draw[densely dotted] (-.5,-3) -- (4.5,-3);
\draw[densely dotted] (-.5,-4) -- (4.5,-4);
\draw[densely dotted] (-.5,-5) -- (4.5,-5);

\draw[very thick] (0,0.5) -- (0,-5.5);
\draw[densely dotted] (1,0.5) -- (1,-5.5);
\draw[densely dotted] (2,0.5) -- (2,-5.5);
\draw[densely dotted] (3,0.5) -- (3,-5.5);
\draw[densely dotted] (4,0.5) -- (4,-5.5);

\node at (0,.5) [above]{$V_1$};
\node at (1,.5) [above]{$V_2$};
\node at (2,.5) [above]{$V_3$};
\node at (3,.5) [above]{$V_4$};
\node at (4,.5) [above]{$V_5$};

\node at (-.5,0) [left]{$H_1$};
\node at (-.5,-1) [left]{$H_2$};
\node at (-.5,-2) [left]{$H_3$};
\node at (-.5,-3) [left]{$H_4$};
\node at (-.5,-4) [left]{$H_5$};
\node at (-.5,-5) [left]{$H_6$};

\end{tikzpicture}
\]}
\end{example}

\begin{corollary}\label{cor:idealK}
    The ideal $K$ of Theorem \ref{pointsplittheorem} is an ideal of a union of lines and ACM points in $\mathbb{P}^{1}\times\mathbb{P}^{1}$. 
\end{corollary}

\begin{proof}
    Since we can consider sets of points as a special case of unions of lines and points, the result follows immediately.
\end{proof}

\begin{theorem}
    Let $I$ be an ideal of a union of lines and ACM points in $\mathbb{P}^{1}\times\mathbb{P}^{1}$. Let $J$ and $K$ be ideals obtained by partitioning the minimal generators of $I$. Then $J\cap K$ is an ideal of a union of lines and ACM points.
\end{theorem}

\begin{proof}
As in Theorem \ref{subsetisACM}, we can assume that $I$ is
an ACM set of points $\mathbb{X}$ 
in $\mathbb{P}^1 \times \mathbb{P}^1$, i.e., $I=I_{\mathbb{X}}$ .  Consequently,
we can assume that there are horizontal and vertical
rulings $H_1,\ldots,H_a$ and $V_1 \ldots,V_b$ that
minimally contain $\mathbb{X}$.  Also, by Theorem
\ref{gens}, all the generators of $I_{\mathbb{X}}$ are 
constructed from products of the $H_i$'s and $V_j$'s.

If $J$ and $K$ are ideals obtained by partitioning the 
minimal generators of $I_{\mathbb{X}}$, then both $J$ and 
$K$ are ideals of unions of lines and ACM sets of points
in $\mathbb{P}^1 \times \mathbb{P}^1$, by Theorem \ref{subsetisACM}.
In particular, there are $i,i',j$ and $j'$ and ACM set of points
$\mathbb{W}$ and $\mathbb{W}'$ such that
\begin{eqnarray*}
J & =& H_1\cdots H_iV_1\cdots V_j I_{\mathbb{W}} ~~\mbox{and}~~ \\
K & =& H_1\cdots H_{i'}V_1 \cdots V_{j'} I_{\mathbb{W'}}.
\end{eqnarray*}
Observe that $\mathbb{W}$ is contained in the rulings
defined by $H_{i+1},\ldots,H_a$ and $V_j,\ldots,V_b$,
while $\mathbb{W}'$ is contained in the rulings defined by
$H_{i'+1},\ldots,H_a$ and $V_{j'+1},\ldots,V_b$.
Now, since 
\begin{eqnarray*}
J & =& \langle H_1 \rangle \cap \cdots \cap \langle H_i \rangle \cap \langle V_1 \rangle \cap \cdots \cap \langle V_j \rangle \cap I_{\mathbb{W}} ~~\mbox{and}~~\\
K & =& \langle H_1 \rangle \cap \cdots \cap
\langle H_{i'}\rangle \cap \langle V_1 \rangle \cap  \cdots \cap
\langle V_{j'}\rangle \cap  I_{\mathbb{W'}},
\end{eqnarray*}
we have
$$J \cap K = \langle H_1 \rangle \cap \cdots \cap
\langle H_{\max\{i,i'\}} \rangle \cap \langle V_1 \rangle
\cap \cdots \cap \langle V_{\max\{j,j'\}} \rangle \cap I_{\mathbb{W}} \cap I_{\mathbb{W}'}.$$
By repeatedly applying Lemma \ref{lem:intersectpts+line}, we get
$$J \cap K = H_1\cdots H_{\max\{i,i'\}}V_1\cdots V_{\max\{j,j'\}}
(I_{\mathbb{W}\setminus Q} \cap I_{\mathbb{W'} \setminus Q})$$
where 
$$Q = H_1 \cup \cdots \cup H_{\max\{i,i'\}} \cup 
V_1 \cup \cdots \cup V_{\max\{j,j'\}}.$$
(We recall that we are abusing notation and letting $H_i$ and
$V_j$ denote both the linear form and the ruling).

By Theorem \ref{lem:intersectpts+line}, both the sets 
$\mathbb{W} \setminus Q$ and $\mathbb{W}' \setminus Q$ are 
ACM sets of points in $\mathbb{P}^1 \times \mathbb{P}^1$.  
Moreover, by construction, both sets of points are contained in
the rulings defined by $H_{\max\{i,i'\}+1},\ldots,H_a$
and $V_{\max\{j,j'\}+1},\ldots,V_b$.

The final step of the proof is to note that $(I_{\mathbb{W}\setminus Q} \cap I_{\mathbb{W'} \setminus Q})$ is the ideal
of an ACM set of points in $\mathbb{P}^1 \times \mathbb{P}^1$
by Lemma \ref{lem:intersect}.  To verify this, one needs to
note that when constructing the ideals
$I_{\mathbb{W}\setminus Q}$ and $ I_{\mathbb{W'} \setminus Q}$,
we will have two ideals which are generated by 
products of $H_{\max\{i,i'\}+1},\ldots,H_a$
and $V_{\max\{j,j'\}+1},\ldots,V_b$ that satisfy the hypotheses
of Lemma \ref{gens=>acm}, i.e., there are four sequences of increasing
integers alongs with these rulings that allow one to construct
$I_{\mathbb{W}\setminus Q}$ and $ I_{\mathbb{W'} \setminus Q}$. \end{proof}

%%%%%%%%%%%%%%%%%%%%%%%%%%%%%%%%%%%%%%%%%%%%%%%%%%%%%%%%%%%%%%%%%%%%
\section{Betti Splittings}\label{sec:Betti}

In this section we discuss the main result of the paper namely, that Betti splittings for ideals of unions of lines and ACM points exist. We begin by reminding the reader of the definition of Betti splitting as introduced by Francisco, H\`a, and Van Tuyl \cite{FHVT2009}.

\begin{definition}
Let $I$ be a (multi-) homogenoues ideal over a (multi-) graded ring $R$, with a set of minimal generators $\mathcal{G}(I)$.
Let $\mathcal{G}(I)=A\sqcup B$ be a partition of the generators of $I$ into non-empty sets $A$ and $B$. Let $J=\langle A\rangle$ and $K=\langle B\rangle$. We call $I = J + K$ a {\it Betti splitting} of $I$ if 
\[\beta_{i,\delta}(I)=\beta_{i,\delta}(J)+\beta_{i,\delta}(K)+\beta_{i-1,\delta}(J\cap K)\] for all non-negative integers $i$ and all (multi-) degrees $\delta$, where we take $\beta_{-1,\delta}(J\cap K)=0$
\end{definition}

From Section \ref{sec:ACM}, we know that when $I$ is an ideal of unions of lines and ACM points and $I=J+K$ is a splitting then $J$, $K$ and $J\cap K$ are all ideals of unions of lines and ACM points. This is enough to characterize the situations where $I=J+K$ is a Betti splitting:   

\begin{theorem}\label{splitting}
    Let $I$ be an ideal of unions of lines and ACM points in $\mathbb{P}^{1}\times\mathbb{P}^{1}$. Let $J$ and $K$ be two non-zero ideals obtained by partitioning the minimal generators of $I$. Then $I=J+K$ is a Betti splitting if and only if $J\cap K$ is principal, that is, $|\mathcal{G}(J\cap K)|=1$.
\end{theorem}

\begin{proof}
    Since the ideals $I$, $J$, $K$ and $J\cap K$ are all ideals of unions of lines and ACM points we know that they have total Betti numbers as follows: Let $n$ be the number of generators of $I$ and let $m$ be the number of generators of $J$. Clearly $\beta_{0,j}(I)=\beta_{0,j}(J)+\beta_{0,j}(K)$ by definition.
    Recall that we have $\beta_{1,j}(I)\leq \beta_{1,j}(J)+\beta_{1,j}(K)+\beta_{0,j}(J\cap K)$. However for ACM ideals we have total Betti numbers $\beta_{1}(I)=\beta_{0}(I)-1$ by Corollary \ref{b1b2}. Thus \[\beta_{0}(I)-1\leq \beta_{0}(J)-1+\beta_{0}(K)-1+\beta_{0}(J\cap K).\] Thus we have equality if and only if $\beta_{0}(J\cap K)=1$.
\end{proof}

This characterization turns out to be quite useful when paired with some additional constructions which will make it clear exactly when it is the case that $|\mathcal{G}(J\cap K)|=1$. We start by providing an extension of Definition \ref{CXVX} to ideals of unions of lines and ACM points. After that we will show how a particular partition of the generators in $\mathcal{G}(I)$ will allow us to determine the cardinality of $\mathcal{G}(J\cap K)$.

\begin{definition}\label{standard_order}
    Using the minimal set of homogeneous generators given in Proposition \ref{gens} we define the {\it standard order} of generators to be the order according to the leading term of each generator given by the lexicographic order $y_{0}>y_{1}>x_{0}>x_{1}$ (i.e. we order the generators according to the order of their leading terms under lexicographic order). We then also define the tuple $\tau$ which has as its entries the generators in the standard order.
\end{definition}

\begin{remark}
    This order is essentially ordering the elements by the number of $V_{B_{i}}$ present in the product which defines the generator. The elements which have more $V_{B_{i}}$'s are first and the element which is entirely a product of the horizontal rulings $H_{A_{1}}\cdots H_{A_{h}}$ is last. 
\end{remark}

\begin{definition}\label{cut_def}
Let $X = \{\tau_1,\ldots,\tau_n\}$ be a set of $n$-distinct elements
and let $\tau = (\tau_1,\ldots,\tau_n)$ be an $n$-tuple of the
elements of $X$ (so $\tau$ can viewed as a fixed ordering of the
elements of $X$).  Let $X = A \sqcup B$ be any partition of $X$
(where $A$ or $B$ is allowed to be empty).  We say $X = A \sqcup B$
is a $0$-partition of $X$ with respect to $\tau$ if either $A$ or
$B$ is empty. 
Otherwise, we say $X = A \sqcup B$ is an 
{\it $r$-cut
partition of $X$ with respect to $\tau$} if there are 
$r$ distinct integers $j_1,\ldots,j_r \in \{1,\ldots,n\}$ with
$j_1 < \cdots < j_r$ such that 
\begin{eqnarray*}
    A &=& \{\tau_1,\tau_2,\ldots,\tau_{j_1}\} \cup \{\tau_{j_2+1},
    \ldots \tau_{j_3}\} \cup \cdots \\
    B &= & \{\tau_{j_1+1},\ldots,\tau_{j_2}\} \cup \{\tau_{j_3+1},
    \ldots,\tau_{j_4}\} \cup \cdots
\end{eqnarray*}
In particular, we call a partition $X=A\sqcup B$, a {\it 1-cut partition of $X$ with respect to $\tau$}, if there is some $j \in \{1,\ldots,n\}$
such that
$A=\{\tau_{i}:i\le j\}$ and $B=\{\tau_{i}:i>j\}$.  
\end{definition}

Informally, we are writing the elements of $X$ in the order
given by $\tau$.  We then place a divider $|$ or ``cut'' between
$\tau_i$ and $\tau_{i+1}$ if $\tau_i \in A$ and $\tau_{i+1} \in B$, or
$\tau_i \in B$ and $\tau_{i+1} \in A$.  Then $X = A \sqcup B$
is an $r$-cut partition if we need $r$ dividers.

\begin{example}
    Consider the distinct elements $X=\{a,b,c,d,e,f,g,h\}$ ordered in
    that natural way, i.e., $\tau=(a,b,c,d,e,f,g,h)$.
    \begin{itemize}
        \item  Consider the partition $X = A \sqcup B = \{a,b,c,f\}
        \sqcup \{d,e,g,h\}$.   We list out the elements of $X$  in the order given by $\tau$, but we put a divider between
        elements if the two elements do not both belong to $A$ or $B$.  In
        this case, we have
        $$a~~b~~c~~|~~d~~e~~|~~f~~|~~g~~h,$$
        since $c \in A$ and $d \in B$, $e \in B$ and $f \in A$,
        and $f \in A$ and $g \in B$.  So our partition is
        a $3$-cut partition of $X$ with respect to $\tau$.
        \item  The partition $X = 
        \{a,b,c,d\} \sqcup \{e,f,g,h\}$ is a 1-cut partition since we only use one divider, that is,
        $a~~b~~c~~d~~|~~e~~f~~g~~h.$
    \end{itemize}
\end{example}

\begin{theorem}\label{cuttheorem}
    Let $I_{\mathbb{W}}$ be an ideal of union of lines and ACM points in $\mathbb{P}^{1}\times\mathbb{P}^{1}$ whose generators have been ordered using the standard order defined in Definition \ref{standard_order} and $\tau$ being the corresponding sequence of generators. 
       Suppose that 
    $\mathcal{G}(I_{\mathbb{W}}) = A\sqcup B$ is an $r$-cut partition
    of the minimal generators of $I_{\mathbb{W}}$ with respect
    to $\tau$. If $J=\langle A\rangle$ and $K=\langle B\rangle$, then $J\cap K$ is minimally generated by $r$ elements.
    \end{theorem}

\begin{proof}
    We proceed by induction on $m$, the number of generators of $I_{\mathbb{W}}$. Suppose $m=1$. In such a case $I_{\mathbb{W}}$ is simply the intersection of lines and only a 0-cut partition is possible, that is, we must have either $K=0$ or $J=0$ and so $J\cap K=0$.  

    Now suppose $m =k+2 > 1$.  As in Theorem \ref{subsetisACM}, we can assume that $I_{\mathbb{W}}$ is an ideal of an ACM set of points $\mathbb{X}$ in $\mathbb{P}^1 \times \mathbb{P}^1$, so we will write $I_{\mathbb{W}}=I_{\mathbb{X}}$. Using the notation of Theorem \ref{gens} we have a minimal generating set given by: \[\mathcal{G}(I_{\mathbb{X}})=\{H_{A_{1}}\cdots H_{A_{h}},V_{B_{1}}\cdots V_{B_{v}}\}\cup \{H_{A_{1}}\cdots H_{A_{i}}\cdot V_{B_{1}}\cdots V_{B_{\alpha_{i+1}}}:\alpha_{i+1}<\alpha_{i}\}.\] We order the generators according to the order given in Definition \ref{standard_order}. We then construct $\tau$ as in Definition \ref{cut_def}, and so have 
    \[\tau=\left(V_{B_{1}}\cdots V_{B_{v}},H_{A_{1}}\cdots H_{A_{h_{1}}}\cdot V_{B_{1}}\cdots V_{B_{v_{1}}},H_{A_{1}}\cdots H_{A_{h_{2}}}\cdot V_{B_{1}}\cdots V_{B_{v_{2}}},\dots ,H_{A_{1}}\cdots H_{A_{h}} \right)\]
    where $\{h_1,\ldots,h_k\}=\{i:\alpha_{i+1}<\alpha_{i}\}$ such that $h_{1}<h_{2}<\cdots <h_{k} < h$ and where $\{v_1,\ldots,v_k\}=\{\alpha_{i+1}:\alpha_{i+1}<\alpha_{i}\}$ such that $v> v_{1}>v_{2}>\cdots >v_{k}$ (note these are defined by the $k$ ``drops'' in $\alpha$ and that $m=k+2$).  Note that when $m=2$, this forces $k=0$, in which
    we only have two generators.
   
    If we remove the final generator of $\tau$, which we will call $\omega=H_{A_{1}}\cdots H_{A_{h}}$, we have that the remaining generators generate an ideal of unions of lines and ACM points by Theorem \ref{subsetisACM}. This ideal, which we denote by $\hat{I}$, is generated by $m-1$ elements, and so by the induction hypothesis if $\mathcal{G}(\hat{I})=A\sqcup B$ is a partition  of a minimal set of generators $\mathcal{G}(\hat{I})$ of $\hat{I}$ with $J=\langle A\rangle$ and $K=\langle B\rangle$, then the number of generators of $J\cap K$ is equal to the cut number of $\mathcal{G}(\hat{I})=A\sqcup B$ with respect to $\hat{\tau}$
    where $\hat{\tau}$ is the first $m-1$ entries of $\tau$.

    There are two possibilities for partitions of the larger set which includes $\omega$. Either $\omega\in A$ or $\omega\in B$. The cut number will be determined by which set the second last element in $\tau$, which we will denote by $\sigma=H_{A_{1}}\cdots H_{A_{h_{k}}}\cdot V_{B_{1}}\cdots V_{B_{v_{k}}}$, belongs to. If $\omega$ and $\sigma$ belong
    to the same set, then the cut number remains the same.  Otherwise,
    the cut number increases by one.
    Without loss of generality assume that $\sigma\in A$.

    If we have the partition $\mathcal{G}(I_{\mathbb{X}})=\left(A\cup\{\omega\}\right)\sqcup B$, then the cut number with respect to $\tau$ is unchanged since $\sigma\in A$. Let $J'=\langle A\cup\{\omega\}\rangle$. We show that $J\cap K=J'\cap K$ and that in particular the minimal number of generators is equal. That $J\cap K\subseteq J'\cap K$ is clear.
    
    So suppose that $L\in J'\cap K$. Then \[L=Q_{1}+c\cdot\omega=Q_{2}\] where $Q_{1}\in J'$ and $c\in R$ and $Q_{2}\in K$. Since $\omega$ has the most $H_{A_{i}}$ in its product it follows that all other elements are divided by $V_{B_{1}}\cdots V_{B_{v_{k}}}$. Further since $\sigma\in A$ by assumption it follows that $V_{B_{1}}\cdots V_{B_{v_{k}}}$ is the greatest common divisor of the generators of $J'$ and hence we may write $Q_{1}=V_{B_{1}}\cdots V_{B_{v_{k}}}\cdot Q'_{1}$ where $V_{j}\nmid Q'_{1}$ for all $j$. Similarly since $\omega\notin B$ we can factor out at least $V_{B_{1}}\cdots V_{B_{v_{k}}}$ from any element of $K$ and so $Q_{2}=V_{B_{1}}\cdots V_{B_{v_{k}}}Q'_{2}$. Hence we have \[L=V_{B_{1}}\cdots V_{B_{v_{k}}}\cdot Q'_{1}+c\cdot\omega=V_{B_{1}}\cdots V_{B_{v_{k}}}Q'_{2}.\] Hence $c\cdot\omega=c'\cdot V_{B_{1}}\cdots V_{B_{v_{k}}}\omega$. We note that $V_{B_{1}}\cdots V_{B_{v_{k}}}\omega= H_{A_{h_{k}+1}}\cdots H_{A_{h}}\sigma\in J$  and so $L\in J\cap K$ and $J\cap K=J'\cap K$ as required. 

    Suppose instead we are in the case where $\omega$ is added to $B$. Here the cut number is increased by one since $\sigma\in A$ by assumption. We need to show that $J\cap K'$ 
    where $K' = \langle B \cup \{ \omega\}\rangle$ has one more generator than $J\cap K$. We will achieve this by showing that \[J\cap K'=J\cap K +\left<V_{B_{1}}\cdots V_{B_{v_{k}}}\omega\right>.\] Again we clearly have that $J\cap K\subseteq J\cap K'$. We recall that $V_{B_{1}}\cdots V_{B_{v_{k}}}\omega=H_{A_{h_{k}+1}}\cdots H_{A_{h}}\sigma$  
    and since $\sigma\in J$ and $\omega\in K'$ we have that $V_{B_{1}}\cdots V_{B_{v_{k}}}\omega\in J\cap K'$ and hence  \[J\cap K +\left<V_{B_{1}}\cdots V_{B_{v_{k}}}\omega\right>\subseteq J\cap K'.\] To prove the reverse inclusion again suppose that $L\in J\cap K'$. Then we may write \[L=Q_{1}=Q_{2}+c\cdot\omega\] where $Q_{1}\in J$, $Q_{2}\in K$ and $c\in R$.

Since $\omega\notin J$ and $\omega\notin K$ by the same argument used in the first case we may write $Q_{1}=V_{B_{1}}\cdots V_{B_{v_{k}}}Q'_{1}$ and $Q_{2}=V_{B_{1}}\cdots V_{B_{v_{k}}}Q'_{2}$ and so have \[L=V_{B_{1}}\cdots V_{B_{v_{k}}}\cdot Q'_{1}=V_{B_{1}}\cdots V_{B_{v_{k}}}Q'_{2}+c\cdot\omega.\] Once again we conclude that $c\cdot\omega=c'\cdot V_{B_{1}}\cdots V_{B_{v_{k}}}\omega$. 
Since $V_{B_{1}}\cdots V_{B_{v_{k}}}\omega=H_{A_{h_{k}+1}}\cdots H_{A_{h}}\sigma\in J$ 
we can subtract $c\cdot\omega$ from $Q_{2}+c\cdot\omega$ and remain in $J$ thus $Q_{2}\in J$ and so $L\in J\cap K +\left<V_{B_{1}}\cdots V_{B_{v_{k}}}\omega\right>$, and so \[J\cap K'=J\cap K +\left<V_{B_{1}}\cdots V_{B_{v_{k}}}\omega\right>\] as required. 

Note that the the generator $V_{B_{1}}\cdots V_{B_{v_{k}}}\omega$ 
is not in $J \cap K$ since $\omega = H_{A_1}\cdots H_{A_h}$, and
no generator in $J\cap K$ is divided by this term by bidegree 
considerations.  Thus we have shown that the number of generators of the intersection corresponds to the cut number and complete the induction. 
\end{proof}

In the following example we will illustrate a more geometric way of viewing the above result relying on the arrangement of unions of lines and ACM sets of points in $\mathbb{P}^{1}\times\mathbb{P}^{1}$.

\begin{example}
Suppose that $\mathbb{X}$ is the ACM set of 20 points
of Example \ref{runex}.  As noted earlier: \small\begin{align*}
    I_{\mathbb{X}} =\langle&V_{1}V_{2}V_{3}V_{4}V_{5}, H_{1}V_{1}V_{2}V_{3}V_{4}, H_{1}H_{2}V_{1}V_{2}V_{3}, H_{1}H_{2}H_{3}H_{4}V_{1}V_{2}, \\
    & H_{1}H_{2}H_{3}H_{4}H_{5}H_{6}V_{1}, H_{1}H_{2}H_{3}H_{4}H_{5}H_{6}H_{7}\rangle
\end{align*}
\normalsize
For this set of generators $\tau$ is given by: \small\[\left( V_{1}V_{2}V_{3}V_{4}V_{5}, H_{1}V_{1}V_{2}V_{3}V_{4}, H_{1}H_{2}V_{1}V_{2}V_{3}, H_{1}H_{2}H_{3}H_{4}V_{1}V_{2}, H_{1}H_{2}H_{3}H_{4}H_{5}H_{6}V_{1}, H_{1}H_{2}H_{3}H_{4}H_{5}H_{6}H_{7}\right).\] 
\normalsize
We will partition the first 5 generators like we describe in the  induction step in the proof of Theorem \ref{cuttheorem}.
In particular, let \small\begin{eqnarray*}
A &=&\{V_{1}V_{2}V_{3}V_{4}V_{5},H_{1}H_{2}V_{1}V_{2}V_{3},H_{1}H_{2}H_{3}H_{4}H_{5}H_{6}V_{1}\}, \\ 
B&=&\{H_{1}V_{1}V_{2}V_{3}V_{4},H_{1}H_{2}H_{3}H_{4}V_{1}V_{2}\}, 
~~\text{and}~~\\
\omega &=& H_{1}H_{2}H_{3}H_{4}H_{5}H_{6}H_{7}.
\end{eqnarray*}
\normalsize
Note that this partition has a cut number of 4. Let $J=\langle A\rangle$ and $K=\langle B\rangle$. From Theorem \ref{subsetisACM} and Lemma \ref{gens=>acm} and the fact that for algebraic sets $X_{1}$ and $X_{2}$ we have $I(X_{1})\cap I(X_{2})=I(X_{1}\cup X_{2})$ we can represent the unions of lines and points as follows:
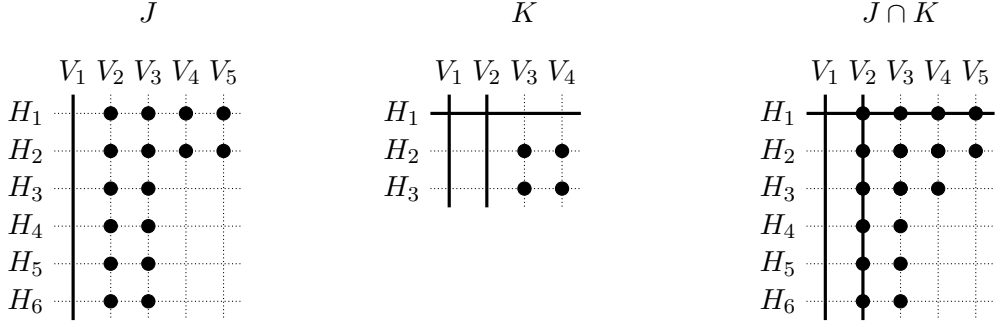
\small
\[
\begin{tikzpicture}[scale =.5]

\draw[fill=black](1,0) circle (5 pt);
\draw[fill=black](2,0) circle (5 pt);
\draw[fill=black](3,0) circle (5 pt);
\draw[fill=black](4,0) circle (5 pt);

\draw[fill=black](1,-1) circle (5 pt);
\draw[fill=black](2,-1) circle (5 pt);
\draw[fill=black](3,-1) circle (5 pt);
\draw[fill=black](4,-1) circle (5 pt);

\draw[fill=black](1,-2) circle (5 pt);
\draw[fill=black](2,-2) circle (5 pt);

\draw[fill=black](1,-3) circle (5 pt);
\draw[fill=black](2,-3) circle (5 pt);

\draw[fill=black](1,-4) circle (5 pt);
\draw[fill=black](2,-4) circle (5 pt);

\draw[fill=black](1,-5) circle (5 pt);
\draw[fill=black](2,-5) circle (5 pt);

\draw[densely dotted] (-.5,0) -- (4.5,0);
\draw[densely dotted] (-.5,-1) -- (4.5,-1);
\draw[densely dotted] (-.5,-2) -- (4.5,-2);
\draw[densely dotted] (-.5,-3) -- (4.5,-3);
\draw[densely dotted] (-.5,-4) -- (4.5,-4);
\draw[densely dotted] (-.5,-5) -- (4.5,-5);

\draw[very thick] (0,0.5) -- (0,-5.5);
\draw[densely dotted] (1,0.5) -- (1,-5.5);
\draw[densely dotted] (2,0.5) -- (2,-5.5);
\draw[densely dotted] (3,0.5) -- (3,-5.5);
\draw[densely dotted] (4,0.5) -- (4,-5.5);

\node at (0,.5) [above]{$V_1$};
\node at (1,.5) [above]{$V_2$};
\node at (2,.5) [above]{$V_3$};
\node at (3,.5) [above]{$V_4$};
\node at (4,.5) [above]{$V_5$};

\node at (-.5,0) [left]{$H_1$};
\node at (-.5,-1) [left]{$H_2$};
\node at (-.5,-2) [left]{$H_3$};
\node at (-.5,-3) [left]{$H_4$};
\node at (-.5,-4) [left]{$H_5$};
\node at (-.5,-5) [left]{$H_6$};

%%%%%%%%%%%%%%%%%%%%%%%%%%%%%%%%%%%%%%%%%%

\draw[fill=black](2+10,-1) circle (5 pt);
\draw[fill=black](3+10,-1) circle (5 pt);

\draw[fill=black](2+10,-2) circle (5 pt);
\draw[fill=black](3+10,-2) circle (5 pt);

\draw[very thick] (-.5+10,0) -- (3.5+10,0);
\draw[densely dotted] (-.5+10,-1) -- (3.5+10,-1);
\draw[densely dotted] (-.5+10,-2) -- (3.5+10,-2);

\draw[very thick] (0+10,0.5) -- (0+10,-2.5);
\draw[very thick] (1+10,0.5) -- (1+10,-2.5);
\draw[densely dotted] (2+10,0.5) -- (2+10,-2.5);
\draw[densely dotted] (3+10,0.5) -- (3+10,-2.5);

\node at (0+10,.5) [above]{$V_1$};
\node at (1+10,.5) [above]{$V_2$};
\node at (2+10,.5) [above]{$V_3$};
\node at (3+10,.5) [above]{$V_4$};

\node at (-.5+10,0) [left]{$H_1$};
\node at (-.5+10,-1) [left]{$H_2$};
\node at (-.5+10,-2) [left]{$H_3$};

%%%%%%%%%%%%%%%%%%%%%%%%%%%%%%%%%%%%%%%%%%%%%%%%%%%%%%%

\draw[fill=black](2+20,-1) circle (5 pt);
\draw[fill=black](3+20,-1) circle (5 pt);

\draw[fill=black](2+20,-2) circle (5 pt);
\draw[fill=black](3+20,-2) circle (5 pt);

\draw[fill=black](1+20,0) circle (5 pt);
\draw[fill=black](2+20,0) circle (5 pt);
\draw[fill=black](3+20,0) circle (5 pt);
\draw[fill=black](4+20,0) circle (5 pt);

\draw[fill=black](1+20,-1) circle (5 pt);
\draw[fill=black](2+20,-1) circle (5 pt);
\draw[fill=black](3+20,-1) circle (5 pt);
\draw[fill=black](4+20,-1) circle (5 pt);

\draw[fill=black](1+20,-2) circle (5 pt);
\draw[fill=black](2+20,-2) circle (5 pt);

\draw[fill=black](1+20,-3) circle (5 pt);
\draw[fill=black](2+20,-3) circle (5 pt);

\draw[fill=black](1+20,-4) circle (5 pt);
\draw[fill=black](2+20,-4) circle (5 pt);

\draw[fill=black](1+20,-5) circle (5 pt);
\draw[fill=black](2+20,-5) circle (5 pt);

\draw[very thick] (-.5+20,0) -- (4.5+20,0);
\draw[densely dotted] (-.5+20,-1) -- (4.5+20,-1);
\draw[densely dotted] (-.5+20,-2) -- (4.5+20,-2);
\draw[densely dotted] (-.5+20,-3) -- (4.5+20,-3);
\draw[densely dotted] (-.5+20,-4) -- (4.5+20,-4);
\draw[densely dotted] (-.5+20,-5) -- (4.5+20,-5);

\draw[very thick] (0+20,0.5) -- (0+20,-5.5);
\draw[very thick] (1+20,0.5) -- (1+20,-5.5);
\draw[densely dotted] (2+20,0.5) -- (2+20,-5.5);
\draw[densely dotted] (3+20,0.5) -- (3+20,-5.5);
\draw[densely dotted] (4+20,0.5) -- (4+20,-5.5);

\node at (0+20,.5) [above]{$V_1$};
\node at (1+20,.5) [above]{$V_2$};
\node at (2+20,.5) [above]{$V_3$};
\node at (3+20,.5) [above]{$V_4$};
\node at (4+20,.5) [above]{$V_5$};

\node at (-.5+20,0) [left]{$H_1$};
\node at (-.5+20,-1) [left]{$H_2$};
\node at (-.5+20,-2) [left]{$H_3$};
\node at (-.5+20,-3) [left]{$H_4$};
\node at (-.5+20,-4) [left]{$H_5$};
\node at (-.5+20,-5) [left]{$H_6$};

\node at (2,2.7) {$J$};
\node at (2+10,2.7) {$K$};
\node at (2+20,2.7) {$J\cap K$};

\end{tikzpicture}
\]
\normalsize
The arrangement of unions of lines and points corresponding to $J\cap K$ is: \small
\[\mathbb{W}_{1}=\{H_{2}\times V_{3},H_{2}\times V_{4},H_{2}\times V_{5},H_{3}\times V_{3},H_{3}\times V_{4},H_{4}\times V_{3},H_{5}\times V_{3},H_{6}\times V_{3},H_{1},V_{1},V_{2}\}.\] 
\normalsize
Then $\alpha_{\mathbb{X}_{\mathbb{W}_{1}}}=(3,2,1,1,1)$
and hence there will be 4 minimal generators 
of $J \cap K$ which agrees with our cut number. Let $\Lambda_{1}=H_{1}V_{1}V_{2}$ denote the products of lines in $\mathbb{W}_{1}$. According to Theorem \ref{gens}, Remark \ref{drops} and Theorem \ref{product} these generators will be: 

\begin{itemize}
    \item[i)] $g_{1}=\Lambda_{1}V_{3}V_{4}V_{5}$, which corresponds to the product of the vertical rulings of $\mathbb{X}_{\mathbb{W}_{1}}$ multiplied by the product of the lines in $\mathbb{W}_{1}$,
    \item[ii)] $g_{2}=\Lambda_{1}H_{2}V_{3}V_{4}$, which corresponds to the generator arising from $(\alpha_{\mathbb{X}_{\mathbb{W}_{1}}})_{1}>(\alpha_{\mathbb{X}_{\mathbb{W}_{1}}})_{2}$ multiplied by the product of the lines in $\mathbb{W}_{1}$,
    \item[iii)] $g_{3}=\Lambda_{1}H_{2}H_{3}V_{3}$, which corresponds to the generator arising from $(\alpha_{\mathbb{X}_{\mathbb{W}_{1}}})_{2}>(\alpha_{\mathbb{X}_{\mathbb{W}_{1}}})_{3}$ multiplied by the product of the lines in $\mathbb{W}_{1}$, and
    \item[iv)] $g_{4}=\Lambda_{1}H_{2}H_{3}H_{4}H_{5}H_{6}$ which corresponds to the product of the horizontal rulings of ${\mathbb{X}_{\mathbb{W}_{1}}}$ multiplied by the product of the lines in $\mathbb{W}_{1}$.
\end{itemize}

The representations of $J'=\langle A\cup \{\omega\}\rangle$, $K=\langle B\rangle$ and $J'\cap K$ are as follows: 
{\small
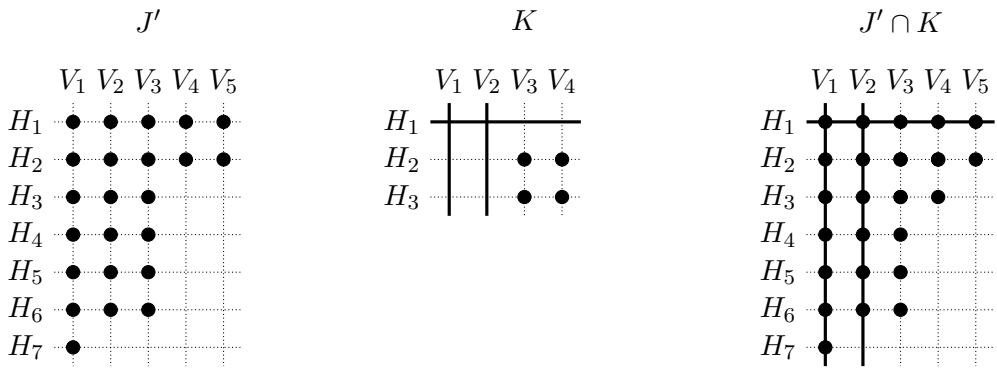
\[
\begin{tikzpicture}[scale =.5]

\draw[fill=black](0,0) circle (5 pt);
\draw[fill=black](1,0) circle (5 pt);
\draw[fill=black](2,0) circle (5 pt);
\draw[fill=black](3,0) circle (5 pt);
\draw[fill=black](4,0) circle (5 pt);

\draw[fill=black](0,-1) circle (5 pt);
\draw[fill=black](1,-1) circle (5 pt);
\draw[fill=black](2,-1) circle (5 pt);
\draw[fill=black](3,-1) circle (5 pt);
\draw[fill=black](4,-1) circle (5 pt);

\draw[fill=black](0,-2) circle (5 pt);
\draw[fill=black](1,-2) circle (5 pt);
\draw[fill=black](2,-2) circle (5 pt);

\draw[fill=black](0,-3) circle (5 pt);
\draw[fill=black](1,-3) circle (5 pt);
\draw[fill=black](2,-3) circle (5 pt);

\draw[fill=black](0,-4) circle (5 pt);
\draw[fill=black](1,-4) circle (5 pt);
\draw[fill=black](2,-4) circle (5 pt);

\draw[fill=black](0,-5) circle (5 pt);
\draw[fill=black](1,-5) circle (5 pt);
\draw[fill=black](2,-5) circle (5 pt);

\draw[fill=black](0,-6) circle (5 pt);

\draw[densely dotted] (0,.5) -- (0,-6.5);
\draw[densely dotted] (1,.5) -- (1,-6.5);
\draw[densely dotted] (2,.5) -- (2,-6.5);
\draw[densely dotted] (3,.5) -- (3,-6.5);
\draw[densely dotted] (4,.5) -- (4,-6.5);

\draw[densely dotted] (-.5,0) -- (4.5,0);
\draw[densely dotted] (-.5,-1) -- (4.5,-1);
\draw[densely dotted] (-.5,-2) -- (4.5,-2);
\draw[densely dotted] (-.5,-3) -- (4.5,-3);
\draw[densely dotted] (-.5,-4) -- (4.5,-4);
\draw[densely dotted] (-.5,-5) -- (4.5,-5);
\draw[densely dotted] (-.5,-6) -- (4.5,-6);

\node at (0,.5) [above]{$V_1$};
\node at (1,.5) [above]{$V_2$};
\node at (2,.5) [above]{$V_3$};
\node at (3,.5) [above]{$V_4$};
\node at (4,.5) [above]{$V_5$};

\node at (-.5,0) [left]{$H_1$};
\node at (-.5,-1) [left]{$H_2$};
\node at (-.5,-2) [left]{$H_3$};
\node at (-.5,-3) [left]{$H_4$};
\node at (-.5,-4) [left]{$H_5$};
\node at (-.5,-5) [left]{$H_6$};
\node at (-.5,-6) [left]{$H_7$};

%%%%%%%%%%%%%%%%%%%%%%%%%%%%%%%%%%%%%%%%%%%%%%%%%%%%%%%%%%%%

\draw[fill=black](2+10,-1) circle (5 pt);
\draw[fill=black](3+10,-1) circle (5 pt);

\draw[fill=black](2+10,-2) circle (5 pt);
\draw[fill=black](3+10,-2) circle (5 pt);

\draw[very thick] (-.5+10,0) -- (3.5+10,0);
\draw[densely dotted] (-.5+10,-1) -- (3.5+10,-1);
\draw[densely dotted] (-.5+10,-2) -- (3.5+10,-2);

\draw[very thick] (0+10,0.5) -- (0+10,-2.5);
\draw[very thick] (1+10,0.5) -- (1+10,-2.5);
\draw[densely dotted] (2+10,0.5) -- (2+10,-2.5);
\draw[densely dotted] (3+10,0.5) -- (3+10,-2.5);

\node at (0+10,.5) [above]{$V_1$};
\node at (1+10,.5) [above]{$V_2$};
\node at (2+10,.5) [above]{$V_3$};
\node at (3+10,.5) [above]{$V_4$};

\node at (-.5+10,0) [left]{$H_1$};
\node at (-.5+10,-1) [left]{$H_2$};
\node at (-.5+10,-2) [left]{$H_3$};

%%%%%%%%%%%%%%%%%%%%%%%%%%%%%%%%%%%%%%%%%%%%%%%%%%%%%

\draw[fill=black](0+20,0) circle (5 pt);
\draw[fill=black](1+20,0) circle (5 pt);
\draw[fill=black](2+20,0) circle (5 pt);
\draw[fill=black](3+20,0) circle (5 pt);
\draw[fill=black](4+20,0) circle (5 pt);

\draw[fill=black](0+20,-1) circle (5 pt);
\draw[fill=black](1+20,-1) circle (5 pt);
\draw[fill=black](2+20,-1) circle (5 pt);
\draw[fill=black](3+20,-1) circle (5 pt);
\draw[fill=black](4+20,-1) circle (5 pt);

\draw[fill=black](0+20,-2) circle (5 pt);
\draw[fill=black](1+20,-2) circle (5 pt);
\draw[fill=black](2+20,-2) circle (5 pt);

\draw[fill=black](0+20,-3) circle (5 pt);
\draw[fill=black](1+20,-3) circle (5 pt);
\draw[fill=black](2+20,-3) circle (5 pt);

\draw[fill=black](0+20,-4) circle (5 pt);
\draw[fill=black](1+20,-4) circle (5 pt);
\draw[fill=black](2+20,-4) circle (5 pt);

\draw[fill=black](0+20,-5) circle (5 pt);
\draw[fill=black](1+20,-5) circle (5 pt);
\draw[fill=black](2+20,-5) circle (5 pt);

\draw[fill=black](0+20,-6) circle (5 pt);

\draw[fill=black](2+20,-1) circle (5 pt);
\draw[fill=black](3+20,-1) circle (5 pt);

\draw[fill=black](2+20,-2) circle (5 pt);
\draw[fill=black](3+20,-2) circle (5 pt);

\draw[very thick] (-.5+20,0) -- (4.5+20,0);
\draw[densely dotted] (-.5+20,-1) -- (4.5+20,-1);
\draw[densely dotted] (-.5+20,-2) -- (4.5+20,-2);
\draw[densely dotted] (-.5+20,-3) -- (4.5+20,-3);
\draw[densely dotted] (-.5+20,-4) -- (4.5+20,-4);
\draw[densely dotted] (-.5+20,-5) -- (4.5+20,-5);
\draw[densely dotted] (-.5+20,-6) -- (4.5+20,-6);

\draw[very thick] (0+20,.5) -- (0+20,-6.5);
\draw[very thick] (1+20,.5) -- (1+20,-6.5);
\draw[densely dotted] (2+20,.5) -- (2+20,-6.5);
\draw[densely dotted] (3+20,.5) -- (3+20,-6.5);
\draw[densely dotted] (4+20,.5) -- (4+20,-6.5);

\node at (0+20,.5) [above]{$V_1$};
\node at (1+20,.5) [above]{$V_2$};
\node at (2+20,.5) [above]{$V_3$};
\node at (3+20,.5) [above]{$V_4$};
\node at (4+20,.5) [above]{$V_5$};

\node at (-.5+20,0) [left]{$H_1$};
\node at (-.5+20,-1) [left]{$H_2$};
\node at (-.5+20,-2) [left]{$H_3$};
\node at (-.5+20,-3) [left]{$H_4$};
\node at (-.5+20,-4) [left]{$H_5$};
\node at (-.5+20,-5) [left]{$H_6$};
\node at (-.5+20,-6) [left]{$H_7$};

\node at (2,2.7) {$J'$};
\node at (2+10,2.7) {$K$};
\node at (2+20,2.7) {$J'\cap K$};

\end{tikzpicture}
\]}
Denote the of unions of lines and points representing $J'\cap K$ by \small
\[\mathbb{W}_{2}=\{H_{2}\times V_{3},H_{2}\times V_{4},H_{2}\times V_{5},H_{3}\times V_{3},H_{3}\times V_{4},H_{4}\times V_{3},H_{5}\times V_{3},H_{6}\times V_{3},H_{1},V_{1},V_{2}\}.\]
\normalsize
Then $\alpha_{\mathbb{X}_{\mathbb{W}_{2}}}=(3,2,1,1,1)$ and hence by Remark \ref{drops} there will be 4 minimal generators which agrees with our cut number. In fact one can see the generators are identical to those of $J\cap K$, a fact which was made use of in the proof of Theorem \ref{cuttheorem}.

Finally we look at the representations of $J=\langle A\rangle$,   and $K'=\langle B\cup \{\omega\}\rangle$ and $J\cap K'$: 
\small
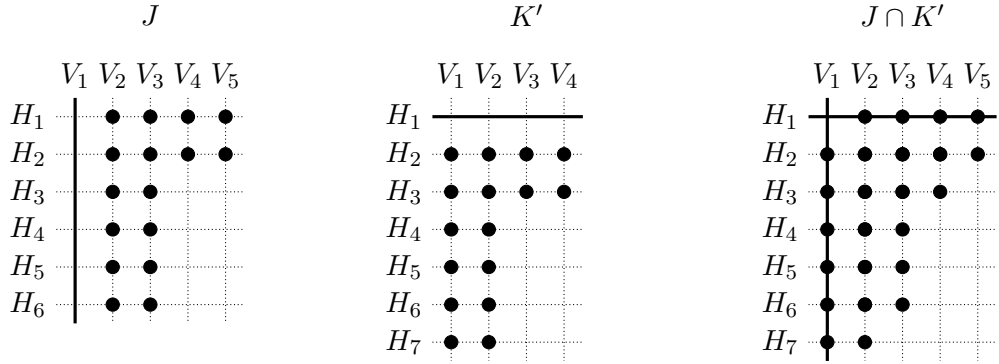
\[
\begin{tikzpicture}[scale =.5]

\draw[fill=black](1,0) circle (5 pt);
\draw[fill=black](2,0) circle (5 pt);
\draw[fill=black](3,0) circle (5 pt);
\draw[fill=black](4,0) circle (5 pt);

\draw[fill=black](1,-1) circle (5 pt);
\draw[fill=black](2,-1) circle (5 pt);
\draw[fill=black](3,-1) circle (5 pt);
\draw[fill=black](4,-1) circle (5 pt);

\draw[fill=black](1,-2) circle (5 pt);
\draw[fill=black](2,-2) circle (5 pt);

\draw[fill=black](1,-3) circle (5 pt);
\draw[fill=black](2,-3) circle (5 pt);

\draw[fill=black](1,-4) circle (5 pt);
\draw[fill=black](2,-4) circle (5 pt);

\draw[fill=black](1,-5) circle (5 pt);
\draw[fill=black](2,-5) circle (5 pt);

\draw[densely dotted] (-.5,0) -- (4.5,0);
\draw[densely dotted] (-.5,-1) -- (4.5,-1);
\draw[densely dotted] (-.5,-2) -- (4.5,-2);
\draw[densely dotted] (-.5,-3) -- (4.5,-3);
\draw[densely dotted] (-.5,-4) -- (4.5,-4);
\draw[densely dotted] (-.5,-5) -- (4.5,-5);

\draw[very thick] (0,0.5) -- (0,-5.5);
\draw[densely dotted] (1,0.5) -- (1,-5.5);
\draw[densely dotted] (2,0.5) -- (2,-5.5);
\draw[densely dotted] (3,0.5) -- (3,-5.5);
\draw[densely dotted] (4,0.5) -- (4,-5.5);

\node at (0,.5) [above]{$V_1$};
\node at (1,.5) [above]{$V_2$};
\node at (2,.5) [above]{$V_3$};
\node at (3,.5) [above]{$V_4$};
\node at (4,.5) [above]{$V_5$};

\node at (-.5,0) [left]{$H_1$};
\node at (-.5,-1) [left]{$H_2$};
\node at (-.5,-2) [left]{$H_3$};
\node at (-.5,-3) [left]{$H_4$};
\node at (-.5,-4) [left]{$H_5$};
\node at (-.5,-5) [left]{$H_6$};

%%%%%%%%%%%%%%%%%%%%%%%%%%%%%%%%%%%%%%%%%%%%%%%%%%%%

\draw[fill=black](0+10,-1) circle (5 pt);
\draw[fill=black](1+10,-1) circle (5 pt);
\draw[fill=black](2+10,-1) circle (5 pt);
\draw[fill=black](3+10,-1) circle (5 pt);

\draw[fill=black](0+10,-2) circle (5 pt);
\draw[fill=black](1+10,-2) circle (5 pt);
\draw[fill=black](2+10,-2) circle (5 pt);
\draw[fill=black](3+10,-2) circle (5 pt);

\draw[fill=black](0+10,-3) circle (5 pt);
\draw[fill=black](1+10,-3) circle (5 pt);

\draw[fill=black](0+10,-4) circle (5 pt);
\draw[fill=black](1+10,-4) circle (5 pt);

\draw[fill=black](0+10,-5) circle (5 pt);
\draw[fill=black](1+10,-5) circle (5 pt);

\draw[fill=black](0+10,-6) circle (5 pt);
\draw[fill=black](1+10,-6) circle (5 pt);

\draw[densely dotted] (0+10,.5) -- (0+10,-6.5);
\draw[densely dotted] (1+10,.5) -- (1+10,-6.5);
\draw[densely dotted] (2+10,.5) -- (2+10,-6.5);
\draw[densely dotted] (3+10,.5) -- (3+10,-6.5);

\draw[very thick] (-.5+10,0) -- (3.5+10,0);
\draw[densely dotted] (-.5+10,-1) -- (3.5+10,-1);
\draw[densely dotted] (-.5+10,-2) -- (3.5+10,-2);
\draw[densely dotted] (-.5+10,-3) -- (3.5+10,-3);
\draw[densely dotted] (-.5+10,-4) -- (3.5+10,-4);
\draw[densely dotted] (-.5+10,-5) -- (3.5+10,-5);
\draw[densely dotted] (-.5+10,-6) -- (3.5+10,-6);

\node at (0+10,.5) [above]{$V_1$};
\node at (1+10,.5) [above]{$V_2$};
\node at (2+10,.5) [above]{$V_3$};
\node at (3+10,.5) [above]{$V_4$};

\node at (-.5+10,0) [left]{$H_1$};
\node at (-.5+10,-1) [left]{$H_2$};
\node at (-.5+10,-2) [left]{$H_3$};
\node at (-.5+10,-3) [left]{$H_4$};
\node at (-.5+10,-4) [left]{$H_5$};
\node at (-.5+10,-5) [left]{$H_6$};
\node at (-.5+10,-6) [left]{$H_7$};

%%%%%%%%%%%%%%%%%%%%%%%%%%%%%%%%%%%%%%%%%%%%%%%%%%%%%%%%%%%%%%%%%%

\draw[fill=black](0+20,-1) circle (5 pt);
\draw[fill=black](1+20,-1) circle (5 pt);
\draw[fill=black](2+20,-1) circle (5 pt);
\draw[fill=black](3+20,-1) circle (5 pt);

\draw[fill=black](0+20,-2) circle (5 pt);
\draw[fill=black](1+20,-2) circle (5 pt);
\draw[fill=black](2+20,-2) circle (5 pt);
\draw[fill=black](3+20,-2) circle (5 pt);

\draw[fill=black](0+20,-3) circle (5 pt);
\draw[fill=black](1+20,-3) circle (5 pt);

\draw[fill=black](0+20,-4) circle (5 pt);
\draw[fill=black](1+20,-4) circle (5 pt);

\draw[fill=black](0+20,-5) circle (5 pt);
\draw[fill=black](1+20,-5) circle (5 pt);

\draw[fill=black](0+20,-6) circle (5 pt);
\draw[fill=black](1+20,-6) circle (5 pt);

\draw[fill=black](1+20,0) circle (5 pt);
\draw[fill=black](2+20,0) circle (5 pt);
\draw[fill=black](3+20,0) circle (5 pt);
\draw[fill=black](4+20,0) circle (5 pt);

\draw[fill=black](1+20,-1) circle (5 pt);
\draw[fill=black](2+20,-1) circle (5 pt);
\draw[fill=black](3+20,-1) circle (5 pt);
\draw[fill=black](4+20,-1) circle (5 pt);

\draw[fill=black](1+20,-2) circle (5 pt);
\draw[fill=black](2+20,-2) circle (5 pt);

\draw[fill=black](1+20,-3) circle (5 pt);
\draw[fill=black](2+20,-3) circle (5 pt);

\draw[fill=black](1+20,-4) circle (5 pt);
\draw[fill=black](2+20,-4) circle (5 pt);

\draw[fill=black](1+20,-5) circle (5 pt);
\draw[fill=black](2+20,-5) circle (5 pt);

\draw[very thick] (0+20,.5) -- (0+20,-6.5);
\draw[densely dotted] (1+20,.5) -- (1+20,-6.5);
\draw[densely dotted] (2+20,.5) -- (2+20,-6.5);
\draw[densely dotted] (3+20,.5) -- (3+20,-6.5);
\draw[densely dotted] (4+20,.5) -- (4+20,-6.5);

\draw[very thick] (-.5+20,0) -- (4.5+20,0);
\draw[densely dotted] (-.5+20,-1) -- (4.5+20,-1);
\draw[densely dotted] (-.5+20,-2) -- (4.5+20,-2);
\draw[densely dotted] (-.5+20,-3) -- (4.5+20,-3);
\draw[densely dotted] (-.5+20,-4) -- (4.5+20,-4);
\draw[densely dotted] (-.5+20,-5) -- (4.5+20,-5);
\draw[densely dotted] (-.5+20,-6) -- (4.5+20,-6);

\node at (0+20,.5) [above]{$V_1$};
\node at (1+20,.5) [above]{$V_2$};
\node at (2+20,.5) [above]{$V_3$};
\node at (3+20,.5) [above]{$V_4$};
\node at (4+20,.5) [above]{$V_5$};

\node at (-.5+20,0) [left]{$H_1$};
\node at (-.5+20,-1) [left]{$H_2$};
\node at (-.5+20,-2) [left]{$H_3$};
\node at (-.5+20,-3) [left]{$H_4$};
\node at (-.5+20,-4) [left]{$H_5$};
\node at (-.5+20,-5) [left]{$H_6$};
\node at (-.5+20,-6) [left]{$H_7$};

\node at (2,2.7) {$J$};
\node at (2+10,2.7) {$K'$};
\node at (2+20,2.7) {$J\cap K'$};

\end{tikzpicture}
\]
\normalsize

We can see from arrangement of unions of lines and points representing $J\cap K'$ is: \small
\begin{align*}
\mathbb{W}_{3} = \{&H_{2}\times V_{2},H_{2}\times V_{3},H_{2}\times V_{4},H_{2}\times V_{5},H_{3}\times V_{2},H_{3}\times V_{3},H_{3}\times V_{4},H_{4}\times V_{2}, \\
&H_{4}\times V_{3},H_{5}\times V_{2},H_{5}\times V_{3},H_{6}\times V_{2},H_{6}\times V_{3},H_{7}\times V_{2},H_{1},V_{1}\}.
\end{align*}
\normalsize
Then $\alpha_{\mathbb{X}_{\mathbb{W}_{3}}}=(4,3,2,2,2,1)$ and hence there will be 5 minimal generators which agrees with the cut number of this partition.  Let $\Lambda_{3}=H_{1}V_{1}$ denote the products of lines in $\mathbb{W}_{3}$. Then proceeding as before, according to Theorem \ref{gens}, Remark \ref{drops} and Theorem \ref{product} these generators will be: 
\begin{eqnarray*}
    h_{1}&=&\Lambda_{3}V_{2}V_{3}V_{4}V_{5}, ~~h_{2}=\Lambda_{3}H_{2}V_{2}V_{3}V_{4},~~ h_{3}=\Lambda_{3}H_{2}H_{3}V_{2}V_{3},\\
    h_{4}&=&\Lambda_{3}H_{2}H_{3}H_{4}H_{5}H_{6}V_{2}, ~~~h_{5}=\Lambda_{3}H_{2}H_{3}H_{4}H_{5}H_{6}H_{7}.
    \end{eqnarray*}
    Note that for $J\cap K'$ the ``new" fifth generator $h_{5}$ is exactly what we saw in the proof of Theorem \ref{cuttheorem}, namely that $h_{5}=V_{1}\omega$.

So in summary we note that the partition corresponding to $J'\cap K$ has 4 cuts while the partition corresponding to $J\cap K'$ has 5 and that the number of minimal generators correspond to the number of cuts as described in the theorem.
\end{example}

\begin{corollary}\label{characterization_cor}
    Let $\mathcal{G}(I)$ be the standard set of generators of $I$ in the standard order. Then a partition of generators is a Betti splitting if and only if the partition is a 1-cut partition of $\mathcal{G}(I)$ with respect to the standard order.
\end{corollary}

\begin{proof}
    This follows immediately from Theorem \ref{splitting} and Theorem \ref{cuttheorem}.
\end{proof}

\begin{corollary}\label{no_betti_point_splittings}
    A point splitting is never a Betti splitting.
\end{corollary}

\begin{proof}
    In the standard order the required generators will never be in the same 1-cut partition.
\end{proof}

%%%%%%%%%%%%%%%%%%%%%%%%%%%%%%%%%%%%%%%%%
\subsection*{Acknowledgments}
%%%%%%%%%%%%%%%%%%%%%%%%%%%%%%%%%%%%%%%%%

Part of the work on this project was
carried out at the Fields Institute in Toronto, Canada
as part of the ``Thematic Program on Commutative Algebra
and Applications''.   All of the authors thank
Fields for funding and for 
providing a wonderful environment in which to work. 
Guardo's research is partially supported by Gnsaga of Indam. Guardo and Keiper's research is partially supported by Progetto Piaceri 2024-26, Università di Catania and by PRIN 2022, “$0$-dimensional schemes, Tensor Theory and applications” – funded by the European Union Next Generation EU, Mission 4, Component 2 -- CUP: E53D23005670006.
Van Tuyl’s research is supported by NSERC Discovery Grant 2024-05299.
%%%%%%%%%%%%%%%%%%%%%%%%%%%%%%%%%%%%%%%%%%%%%%%%%%%%%%%%%

\bibliographystyle{abbrv}
\bibliography{references}

\end{document}